\documentclass[a4paper,reqno]{amsart}

\usepackage{amssymb}
\usepackage{latexsym}
\usepackage{amsmath}
\usepackage[mathcal]{euscript}
\usepackage{bbm}

      \def\sL{{\mathfrak L}}

   \def\dN{{\mathbb N}}   
      \def\dR{{\mathbb R}}

      \def\cC{{\mathcal C}}
\def\cD{{\mathcal D}}      
\def\cG{{\mathcal G}}   \def\cH{{\mathcal H}}   
   \def\cK{{\mathcal K}}   \def\cL{{\mathcal L}}
      \def\cO{{\mathcal O}}

\def\cal H{{\mathcal H}}

\def\R{\mathbb{R}}
\def\C{\mathbb{C}}

\def\N{\mathbb{N}}

\def\ran{{\text{\rm ran\,}}}
\def\dom{{\text{\rm dom\,}}}

\def\phi{\varphi}

\def\eps{\varepsilon}

\DeclareMathOperator{\clsp}{clsp}

\DeclareMathOperator{\Res}{Res}

\DeclareMathOperator{\clac}{cl_{ac}}
\DeclareMathOperator{\clc}{cl_{c}}

\DeclareMathOperator{\supp}{supp}
\DeclareMathOperator{\cl}{cl}
\DeclareMathOperator{\Real}{Re}
\DeclareMathOperator{\Imag}{Im}

\newtheorem{theorem}{Theorem}[section]
\newtheorem*{thm*}{Theorem}
\newtheorem{proposition}[theorem]{Proposition}
\newtheorem{corollary}[theorem]{Corollary}
\newtheorem{lemma}[theorem]{Lemma}
\newtheorem{assumption}[theorem]{Assumption}

\theoremstyle{definition}
\newtheorem{definition}[theorem]{Definition}
\newtheorem{example}[theorem]{Example}

\newtheorem{remark}[theorem]{Remark}

\numberwithin{equation}{section}

\title[]{Spectral analysis of selfadjoint elliptic differential operators, Dirichlet-to-Neumann maps, and abstract Weyl functions}

\author{Jussi Behrndt \and Jonathan Rohleder}

\address{Institut f\"ur Numerische Mathematik \\
Technische Universit\"at Graz \\
Steyrergasse 30\\
A-8010 Graz\\
Austria}
\email{behrndt@tugraz.at \and rohleder@tugraz.at}

\begin{document}

\begin{abstract}
The spectrum of a selfadjoint second order elliptic differential operator in $L^2(\dR^n)$ is described in terms of the limiting behavior of 
Dirichlet-to-Neumann maps, which arise in a multi-dimensional Glazman decomposition and correspond to an interior and an 
exterior boundary value problem. This leads to PDE analogs of renowned facts in spectral theory of ODEs. 
The main results in this paper are first derived in the more abstract context 
of extension theory of symmetric operators and corresponding Weyl functions, 
and are applied to the PDE setting afterwards. 

%
%
\end{abstract}

\maketitle

\section{Introduction}
The Titchmarsh--Weyl function is an indispensable tool in direct and inverse spectral theory of ordinary
differential operators and more general systems of ordinary differential equations; see the 
classical monographs~\cite{CL55,T62} and~\cite{B01,DK99,GS96,GS00-1,GP87,HS81,KST12,LW11,S96,S99} for a small 
selection of more recent contributions. For a singular second order
Sturm--Liouville differential operator of the form $\sL_+=- \frac{d^2}{d x^2} + q_+$ on $\R_+$ 
with a real-valued, bounded potential $q_+$ the Titchmarsh--Weyl function $m_+$ can be defined as
\begin{equation}\label{mfct}
m_+(\lambda)=\frac{f_\lambda' (0)}{f_\lambda (0)}, \qquad \lambda \in \C \setminus \R,
\end{equation}
where $f_\lambda$ is a square-integrable solution of $\sL_+ f = \lambda f$ on $\R_+$; cf.~\cite{T62,W10}.
The function  $m_+ : \C \setminus \R \to \C$ belongs to the class of Nevanlinna (or Riesz--Herglotz) functions
and it is a celebrated fact that it reflects the complete spectral properties 
of the selfadjoint realizations of $\sL_+$ in $L^2 (\R_+)$. E.g. 
the eigenvalues of the Dirichlet realization $A_D$ are precisely those $\lambda \in \R$, 
where $\lim_{\eta \searrow 0} i \eta m_+ (\lambda + i \eta) \neq 0$, the isolated eigenvalues among them coincide with the poles of $m_+$, and 
the absolutely continuous spectrum of $A_D$ (roughly speaking) consists of all $\lambda$ with the property 
$0 < \Imag m_+ (\lambda + i 0) < + \infty$. 

If $\sL=- \frac{d^2}{d x^2} + q$ is a 
singular Sturm-Liouville expression on $\R$ with $q$ real-valued and bounded, it is most natural 
to use decomposition methods of Glazman type for the analysis of the corresponding selfadjoint operator in $L^2(\dR)$; cf.~\cite{G65}.
More precisely, the restriction of $\sL$ to $\R_+$ gives rise to the Titchmarsh--Weyl function $m_+$ in 
\eqref{mfct}, and similarly a Titchmarsh--Weyl function $m_-$ associated to the restriction of $\sL$ to $\R_-$ is defined.
In that case usually the functions
\begin{equation}\label{mfct2}
m(\lambda)=-\bigl(m_+(\lambda)+m_-(\lambda)\bigr)^{-1}\quad\text{and}\quad 
\widetilde m(\lambda)= \begin{pmatrix} - m_+(\lambda) & 1 \\ 1 & m_-(\lambda)^{-1}  \end{pmatrix}^{-1}
\end{equation}
are employed for the description of the spectrum. Whereas the scalar function $m$ seems to be more convenient
it will in general not contain the complete spectral data, a drawback that is overcome when using the $2\times 2$-matrix
function $\widetilde m$. Some of these observations were already made in \cite{K63,T62}, similar ideas can also be found 
in \cite{HSW00,HS83,K89} for Hamiltonian systems, and more recently in an abstract operator theoretical framework in \cite{DHMS00,DHMS09},
see also \cite{BLu07,BLT13}.

One of the main objectives of this paper is to extend the classical spectral analysis of ordinary differential operators 
via the Titchmarsh--Weyl functions in~\eqref{mfct2} to the multidimensional setting.
For this consider the second order partial differential expression 
\begin{align}\label{eq:diffexprIntro}
 \cL  = - \sum_{j, k = 1}^n \frac{\partial}{\partial x_j} a_{jk} \frac{\partial}{\partial x_k} +
 \sum_{j = 1}^n \left( a_j \frac{\partial}{\partial x_j} - \frac{\partial}{\partial x_j} \overline{a_j} \right) + a
\end{align}
with smooth, bounded coefficients $a_{jk}, a_j : \R^n \to \C$ and $a: \R^n \to \R$ bounded, 
and assume that $\cL$ is formally symmetric and uniformly elliptic on $\R^n$. Let $A$ be the selfadjoint operator associated to \eqref{eq:diffexprIntro} in $L^2(\R^n)$.
Our main goal is to describe the spectral data of $A$, that is, isolated and embedded eigenvalues, 
continuous, absolutely continuous and singular continuous spectral points, in terms of the limiting behaviour of appropriate 
multidimensional counterparts of the functions in \eqref{mfct2}. Note first that
the multidimensional analogue of the Titchmarsh--Weyl function \eqref{mfct}
is the Dirichlet-to-Neumann map, and in order to define suitable analogues of the functions in \eqref{mfct2} we proceed 
as follows: Split $\R^n$ into a bounded domain $\Omega_{\rm i}$ with smooth boundary $\Sigma$ and let
$\Omega_{\rm e} = \R^n \setminus \overline{\Omega_{\rm i}}$ be the exterior of $\Omega_{\rm i}$.
For $\lambda \in \C \setminus \R$ the 
Dirichlet-to-Neumann maps for $\cL$ in $\Omega_{\rm i}$ and $\Omega_{\rm e}$, respectively, on the compact interface $\Sigma$ are given by
\begin{align*}
 \Lambda_{\rm i} (\lambda) u_{\lambda, \rm i} |_\Sigma := \frac{\partial u_{\lambda, \rm i}}{\partial \nu_{\cL_{\rm i}}} \Big|_\Sigma
 \quad\text{and}\quad
 \Lambda_{\rm e} (\lambda) u_{\lambda, \rm e} |_\Sigma := \frac{\partial u_{\lambda, \rm e}}{\partial \nu_{\cL_{\rm e}}} \Big|_\Sigma,
 \qquad \lambda \in \C \setminus \R,
\end{align*}
where $u_{\lambda, j} \in H^2 (\Omega_j)$ solve $\cL u_{\lambda, j} = \lambda u_{\lambda,j}$, $j = \rm i, e$, and 
$u_{\lambda, j}|_\Sigma$ and $\frac{\partial u_{\lambda,j}}{\partial \nu_{\cL_j}} |_\Sigma$ denote the trace 
and the conormal derivative, respectively; cf.~Section~\ref{41} for further details. Both functions $\Lambda_{\rm i}$ and $\Lambda_{\rm e}$ are viewed
as operator-valued functions in $L^2(\Sigma)$ defined on the dense subspace $H^{3/2}(\Sigma)$. The multidimensional counterparts of
the functions in \eqref{mfct2} are
\begin{align}\label{eq:WeylSchreodIntro}
 M (\lambda) = \big( \Lambda_{\rm i} (\lambda) + \Lambda_{\rm e} (\lambda) \big)^{-1}
\quad\text{and}\quad
 \widetilde M(\lambda)= \begin{pmatrix} \Lambda_{\rm i}(\lambda) & 1 \\ 1 & -\Lambda_{\rm e}(\lambda)^{-1}
 \end{pmatrix}^{-1} 
\end{align}
(the differences in the signs are due to the definition of the conormal derivative, where the normals of $\Omega_{\rm i}$ and $\Omega_{\rm e}$
point into opposite directions). Observe that, in contrast to the one-dimensional situation described above,
$\R^n$ is split into a bounded domain and an unbounded domain. This yields that $\Lambda_{\rm i}$ is meromorphic, which in turn essentially allows us to 
give an almost complete characterization
of the spectrum of $A$ with the function $M$ in \eqref{eq:WeylSchreodIntro} in Theorem~\ref{thm:eigenSchroed}; the only possible
spectral points that cannot be detected with $M$ are eigenvalues of $A$ with vanishing traces on $\Sigma$, and possible
accumulation points of such eigenvalues. A complete picture of the spectrum of $A$ in terms of the limiting behaviour of Dirichlet-to-Neumann maps
is obtained with help of the $2\times 2$-block operator matrix function $\widetilde M$ in  
\eqref{eq:WeylSchreodIntro} in Theorem~\ref{thm:eigenSchroedDecoup}.

We mention that in connection with Schr\"{o}dinger operators in $\R^3$ the function $M$ in \eqref{eq:WeylSchreodIntro} was already used 
in \cite{AP04} for the extension of a classical 
convergence property of the Titchmarsh--Weyl function to the three-dimensional case, see also \cite{BLL13,BLL13-2,R09}. We also remark that for Schr\"odinger operators on exterior domains with $C^2$-boundaries
the connection of the spectrum to the limits of the Dirichlet-to-Neumann map was already investigated by the authors in~\cite{BR13}.

In this paper our approach to Titchmarsh--Weyl functions and their connection to spectral properties of corresponding selfadjoint 
differential operators is more abstract and of general nature, based on the concepts of (quasi) boundary triplets and their Weyl functions. 
Recall first that for a symmetric operator $S$ in a Hilbert space $\cH$ a 
boundary triple $\{ \cG, \Gamma_0, \Gamma_1 \}$ consists of a ``boundary space'' $\cG$ and two
linear mappings $\Gamma_0, \Gamma_1 : \dom S^* \to \cG$, which satisfy an abstract Green identity 
\begin{equation}\label{gi}
 (S^* f, g)_\cH - (f, S^* g)_\cH = (\Gamma_1 f, \Gamma_0 g)_\cG - (\Gamma_0 f, \Gamma_1 g)_\cG, \quad f, g \in \dom S^*,
\end{equation}
and a maximality condition. The corresponding Weyl function $M$ is defined as 
\begin{equation}\label{wf}
 M (\lambda) \Gamma_0 f_\lambda = \Gamma_1 f_\lambda, \qquad \lambda \in \C \setminus \R,
\end{equation}
where $f_\lambda \in \cH$ solves the equation $S^* f = \lambda f$; the values $M(\lambda)$ of the Weyl function $M$ are 
bounded operators
in the Hilbert space $\cG$.
The example of the Sturm--Liouville expression $\sL_+$ in the beginning of the introduction fits into this scheme: 
There $\cH = L^2 (\R_+)$,~$S$ is the minimal operator 
associated with the differential expression $\mathfrak{L}_+$ in $L^2 (\R_+)$, $\cG = \C$,
and the mappings $\Gamma_0, \Gamma_1$ are given by
\begin{align*}
 \Gamma_0 f = f (0) \quad \text{and} \quad \Gamma_1 f = f' (0), \qquad f \in \dom S^*,
\end{align*}
where $S^*$ is the maximal operator associated with $\sL_+$ in $L^2 (\R_+)$. Then
the corresponding Weyl function is $m_+$ in \eqref{mfct}, 
the selfadjoint Dirichlet operator $A_D$ coincides with $S^* \upharpoonright \ker \Gamma_0$, and the
spectrum can be described with the help of the limits of the Weyl function. 
The correspondence between the spectrum of the particular selfadjoint extension $A_0 := S^* \upharpoonright \ker \Gamma_0$ and the 
limits of the Weyl function is not a special feature of the boundary triple for 
the above Sturm--Liouville equation. In fact, it holds as soon as the symmetric 
restriction $S$ (and, thus, the boundary mappings $\Gamma_0$ and $\Gamma_1$) is chosen properly. 
More abstract considerations from \cite{DM91,KL73,KL77,LT77} yield that the operator $A_0$ (and hence its spectrum) 
is determined up to unitary equivalence by 
the Weyl function if and only if 
the symmetric operator $S$ is simple or completely non-selfadjoint, that is, there 
exists no nontrivial subspace of $\cH$ which reduces $S$ to a selfadjoint operator. This condition can be reformulated equivalently as
\begin{align}\label{eq:simpleIntro}
 \cH = \clsp \bigl\{ \gamma (\nu) g : \nu \in \C \setminus \R, \, g \in \cG \bigr\},
\end{align}
where $\gamma (\nu)=(\Gamma_0\upharpoonright\ker(S^*-\nu))^{-1}$ is the so-called $\gamma$-field and $\clsp$ denotes the closed linear span; cf.~\cite{K49}. 
Under the assumption that $S$ is simple a description of the absolutely continuous and singular continuous spectrum 
in the framework of boundary triples and their Weyl functions was given in \cite{BMN02}; for more recent related work see also~\cite{BGW09,BHMNW09,BMNW08,BGP08,HMM13,M10,MN11,P13,R07,SW13}.

The concept of boundary triples and their Weyl functions was extended in~\cite{BL07} in such a way that it is conveniently applicable
to PDE problems. For that one defines boundary mappings $\Gamma_0,\Gamma_1$ on a suitable, smaller subset of the domain of the maximal 
operator and 
requires Green's identity \eqref{gi} only to hold on this subset; the definition of the Weyl function associated to such a quasi boundary triple $\{ \cG, \Gamma_0, \Gamma_1 \}$ is as in \eqref{wf}, except that only solutions in the domain of the boundary maps are used; cf. Section~\ref{21}. For the second order elliptic operator $\cL$ in~\eqref{eq:diffexprIntro} restricted to the smooth domain 
$\Omega_{\rm i}\subset\R^n$ one may choose $\cG=L^2(\Sigma)$,
\begin{align*}
 \Gamma_0 u = u\vert_{\partial\Omega_{\rm i}} \quad \text{and} \quad \Gamma_1 u = 
 -\frac{\partial u}{\partial\nu_{\cL_{\rm i}}}\Big|_{\partial\Omega_{\rm i}}, 
 \qquad u \in H^2(\Omega),
\end{align*}
in which case the corresponding Weyl function is (minus) the Dirichlet-to-Neumann map $-\Lambda_{\rm i}$.
Based on
orthogonal couplings of symmetric operators and extending abstract ideas in \cite{DHMS00} 
also the functions $M$ and $\widetilde M$ in \eqref{eq:WeylSchreodIntro} can be interpreted as Weyl functions associated
to properly chosen quasi boundary triples; e.g., $M$ corresponds to the pair of boundary mappings 
\begin{equation}\label{couple}
 \Gamma_0 u = \frac{\partial u_{\rm i}}{\partial \nu_{\cL_{\rm i}}}
 \Big|_\Sigma + \frac{\partial u_{\rm e}}{\partial \nu_{\cL_{\rm e}}} \Big|_\Sigma, \quad \Gamma_1 u = u |_\Sigma, \qquad u = u_{\rm i} \oplus u_{\rm e}, \quad u_{\rm i} |_\Sigma = u_{\rm e} |_\Sigma,
\end{equation}
where $u_j \in H^2 (\Omega_j)$, $j = \rm i, e$. Moreover, 
$\ker \Gamma_0$ is the domain of the unique selfadjoint operator $A$ associated 
with $\cL$ in $L^2 (\R^n)$. When trying to link the spectral properties of $A$
to the limiting behavior of the function $M$ it is necessary to extend the known results 
for boundary triples to the more general notion of quasi boundary triples. Moreover, a subtle 
difficulty arises: The symmetric operator $S$ corresponding to the boundary mappings in \eqref{couple} may possess eigenvalues and thus in general is not
simple. 

In the abstract part of
the present paper we show how this difficulty can be overcome. In the general setting of quasi boundary triples and their 
Weyl functions we show that a local simplicity condition on an open interval (or, more generally,  a Borel set) $\Delta \subset \R$ 
suffices to characterize the spectrum of $A_0$ in $\Delta$. To be more specific, we assume that
\begin{align}\label{eq:localSimpleIntro}
 E (\Delta) \cH = \clsp \big\{ E (\Delta) \gamma (\nu) g : \nu \in \C \setminus \R, \, g \in \ran \Gamma_0 \big\},
\end{align}
where $E (\Delta)$ denotes the spectral projection of $A_0 = S^* \upharpoonright \ker \Gamma_0$ on $\Delta$; this is a local version of the condition~\eqref{eq:simpleIntro}. 
Under this assumption we provide characterizations of the isolated and embedded eigenvalues and the corresponding eigenspaces, as well as the continuous, 
absolutely continuous and singular continuous spectrum of $A_0$ in~$\Delta$ in terms of the limits of $M (\lambda)$ when $\lambda$ approaches the real axis. 
For instance, we prove that the eigenvalues of $A_0$ in $\Delta$ are those $\lambda$, where $\lim_{\eta \searrow 0} i \eta M (\lambda + i \eta) g \neq 0$ 
for some $g \in \ran \Gamma_0$, and that the absolutely continuous spectrum of $A_0$ can be characterized by means of the points 
$\lambda$ where $0 < \Imag (M (\lambda + i 0) g, g)_\cG < \infty$. Moreover, we prove inclusions and provide conditions for the absence of 
singular continuous spectrum. 
Afterwards we apply the obtained results to the selfadjoint elliptic differential operator associated to $\cL$ in \eqref{eq:diffexprIntro} in $L^2 (\R^n)$. 
We prove that, despite the fact that the underlying symmetric 
operator fails to be simple in general, the whole absolutely continuous spectrum of $A_0$ can be recovered from the mapping $M$ in~\eqref{eq:WeylSchreodIntro}. 
Moreover, we prove that the eigenvalues of $A_0$ and the corresponding eigenfunctions can be characterized by limiting properties of $M$ as 
far as the eigenfunctions do not vanish on the interface $\Sigma$. A complete picture of the spectrum of $A_0$ is obtained when
using the function $\widetilde M$ in~\eqref{eq:WeylSchreodIntro}.

This paper is organized in the following way. In Section~2 we recall the basic facts on quasi boundary triples and corresponding Weyl functions and discuss the local simplicity property~\eqref{eq:localSimpleIntro} in detail. In Section~3 the connection between the spectra of selfadjoint operators and corresponding abstract Weyl functions is investigated. Section~4 contains the application of the abstract results to the mentioned PDE problems.

Finally, let us fix some notation. For a selfadjoint operator $A$ in a Hilbert space $\cH$ we denote by $\sigma (A)$ 
($\sigma_{\rm p} (A), \sigma_{\rm c} (A)$, $\sigma_{\rm ac} (A)$, $\sigma_{\rm sc} (A)$, $\sigma_{\rm s} (A)$, 
respectively) the spectrum (set of eigenvalues, continuous, absolutely continuous, singular continuous, singular 
spectrum, respectively) of $A$ and by $\rho (A) = \C \setminus \sigma (A)$ its resolvent set.

\section{Quasi boundary triples, associated Weyl functions, and a local simplicity condition}

In this preliminary section we first recall the concepts of quasi boundary triples, their $\gamma$-fields and their Weyl functions. Afterwards we discuss a local simplicity property of symmetric operators, which will be assumed to hold in most of the results of Section~\ref{sec:abstr}.

\subsection{Quasi boundary triples}\label{21}

The notion of quasi boundary triples was introduced in~\cite{BL07}
as a generalization of the notions of boundary triples and generalized boundary triples, see~\cite{DHMS06,DM91,DM95,GG91,K75}. 
The basic definition is the following.

\begin{definition}\label{def:qbt}
Let $S$ be a closed, densely defined, symmetric operator in a separable Hilbert 
space $\cH$ and let $T \subset S^*$ be an operator whose closure coincides with $S^*$, i.e., $\overline T = S^*$. A triple $\{ \cG, \Gamma_0, \Gamma_1\}$ 
consisting of a Hilbert space $\cG$ and 
two linear mappings $\Gamma_0, \Gamma_1 : \dom T \to \cG$ is called a 
{\em quasi boundary triple} for $S^*$ if the following conditions are 
satisfied.
\begin{enumerate}
  \item The range of the mapping $\Gamma:=(\Gamma_0, \Gamma_1)^\top:\dom T\rightarrow \cG \times \cG$ is dense.
  \item The identity
   \begin{align}\label{eq:absGreen}
    (T u, v)_\cH - (u, T v)_\cH = (\Gamma_1 u, \Gamma_0 v)_\cG - 
(\Gamma_0 u, \Gamma_1 v)_\cG
   \end{align}
   holds for all $u, v \in \dom T$.
  \item The operator $A_0 := T \upharpoonright \ker \Gamma_0$ is 
selfadjoint in $\cH$.
\end{enumerate}
\end{definition}

In the following we suppress the indices in the scalar products and simply write $(\cdot, \cdot)$, when no confusion can arise.

We recall some facts on quasi boundary triples, which can be found in \cite{BL07,BL11}.
Let $S$ be a closed, densely defined, symmetric operator in 
$\cH$. A quasi boundary triple $\{ \cG, \Gamma_0, \Gamma_1\}$ for $S^*$ exists if and only if the defect numbers of $S$ are equal. What we will use frequently is that if $\{ \cG, \Gamma_0, \Gamma_1\}$ is a quasi boundary triple for $S^*$ then $\dom S=\ker\Gamma_0\cap\ker\Gamma_1$.
Recall also that a quasi boundary triple with the additional property $\ran(\Gamma_0,\Gamma_1)^\top=\cG\times\cG$ becomes an (ordinary) boundary triple and that, in particular, in this case the boundary mappings $\Gamma_0,\Gamma_1$ are defined on $\dom S^*$ and \eqref{eq:absGreen} holds
with $T$ replaced by $S^*$. In particular, in the case of finite defect numbers the notions of quasi boundary triples and (ordinary) boundary triples coincide. 
For more details
on quasi boundary triples we refer to \cite{BL07,BL11}.

In order to prove that a triple $\{ \cG, \Gamma_0, \Gamma_1\}$ is a 
quasi boundary triple for the adjoint $S^*$ of a given symmetric 
operator $S$ it is not necessary to know $S^*$ explicitly, as the 
following useful proposition shows; cf.~\cite[Theorem~2.3]{BL07} for a proof.

\begin{proposition}\label{prop:ratetheorem}
Let $T$ be a linear operator in a separable Hilbert space $\cH$, let $\cG$ be a 
further Hilbert space, and let $\Gamma_0, \Gamma_1 : \dom T \to \cG$ be 
linear mappings which satisfy the following conditions.
\begin{enumerate}
  \item The range of the map $\Gamma = ( \Gamma_0, \Gamma_1 )^\top : \dom T \to 
\cG \times \cG$ is dense in $\cG\times\cG$ and $\ker \Gamma$ is dense in $\cH$.
  \item The identity~\eqref{eq:absGreen} holds for all $u,v\in\dom T$.
  \item There exists a selfadjoint restriction $A_0$ of $T$ in $\cH$ with 
$\dom A_0 \subset \ker \Gamma_0$.
\end{enumerate}
Then $S := T \upharpoonright \ker \Gamma$ is a closed, densely defined, 
symmetric operator in $\cH$, $\overline T = S^*$ holds, and $\{ \cG, 
\Gamma_0, \Gamma_1 \}$ is a quasi boundary triple for $S^*$ with $T 
\upharpoonright \ker \Gamma_0 = A_0$. 
\end{proposition}

\subsection{$\gamma$-fields and Weyl functions}

Let $S$ be a closed, densely defined, symmetric operator in 
$\cH$ and let $\{\cG,\Gamma_0,\Gamma_1\}$ be a quasi boundary triple for $\overline T=S^*$ with $A_0=T\upharpoonright\ker\Gamma_0$.
In order to define the $\gamma$-field and the Weyl function corresponding to $\{\cG,\Gamma_0,\Gamma_1\}$
note that the direct sum decomposition
\begin{align*}
  \dom T = \dom A_0 \dotplus \ker (T - \lambda) = \ker \Gamma_0 \dotplus 
\ker (T - \lambda)
\end{align*}
holds for each $\lambda \in \rho (A_0)$ and that, in particular, the 
restriction of $\Gamma_0$ to $\ker (T - \lambda)$ is injective. The following definition
is formally the same as for ordinary and generalized boundary triples.

\begin{definition}\label{def:gammaweyl}
Let $\{ \cG, \Gamma_0, \Gamma_1\}$ be a quasi boundary 
triple for $\overline T=S^*$ and let $A_0=T\upharpoonright\ker\Gamma_0$.
Then the {\em $\gamma$-field} $\gamma$ and the {\em 
Weyl function} $M$ associated with $\{\cG, \Gamma_0, \Gamma_1\}$ are given by
\begin{align*}
  \gamma (\lambda) = \big( \Gamma_0 \upharpoonright \ker (T - \lambda) 
\big)^{-1} \quad \text{and} \quad M (\lambda) = \Gamma_1 \gamma (\lambda),\quad\lambda\in\rho(A_0),
\end{align*}
respectively.
\end{definition}

It follows immediately from the definition that for each $\lambda \in 
\rho (A_0)$ the operator $M (\lambda)$ satisfies the equality
\begin{align*}
  M (\lambda) \Gamma_0 u_\lambda = \Gamma_1 u_\lambda, \qquad u_\lambda 
\in \ker (T - \lambda),
\end{align*}
and that $\ran \gamma (\lambda) = \ker (T - \lambda)$ holds.
We summarize some properties of the $\gamma$-field and the Weyl 
function. For the proofs of items (i)-(iv) in the next lemma we refer to~\cite[Proposition~2.6]{BL07}, item (v) is a simple consequence of (ii) and (iii).

\begin{lemma}\label{lem:gammaWeylProp}
Let $\{ \cG, \Gamma_0, \Gamma_1 \}$ be a quasi boundary triple for $\overline T = S^*$ with $\gamma$-field $\gamma$ and Weyl function $M$ and let $A_0=T\upharpoonright\ker\Gamma_0$. Then for $\lambda, \mu,\nu \in \rho (A_0)$ the following 
assertions hold.
\begin{enumerate}
  \item $\gamma (\lambda)$ is a bounded operator from 
$\cG$ to $\cH$ defined on the dense subspace $\ran\Gamma_0$. The adjoint $\gamma (\lambda)^* : \cH \to \cG$ is defined on $\cH$ and is bounded. It is
given by
  \begin{align*}
   \gamma (\lambda)^* = \Gamma_1 (A_0 - \overline \lambda)^{-1}.
  \end{align*}
  \item The identity
  \begin{align*}
   \gamma (\lambda)g = \left( I + (\lambda - \mu) (A_0 - \lambda)^{-1} 
\right) \gamma (\mu)g
  \end{align*}
  holds for all $g\in\ran\Gamma_0$.
  \item The $\gamma$-field and the Weyl function are connected via
  \begin{align*}
   (\lambda - \overline \mu) \gamma (\mu)^* \gamma (\lambda)g = M 
(\lambda)g - M (\mu)^*g,\qquad g\in \ran\Gamma_0,
  \end{align*}
  and $M (\overline \lambda) \subset M (\lambda)^*$ holds.
  \item $M (\lambda)$ is an operator in $\cG$ defined on the dense subspace $\ran\Gamma_0$ and satisfies
  \begin{equation}\label{eq:Mformula}
   \begin{split}
    \qquad M (\lambda)g &  = \Real M (\mu)g \\  
     & \qquad + \gamma (\mu)^* \left( (\lambda - \Real 
\mu) + (\lambda - \mu) (\lambda - \overline \mu) (A_0 - \lambda)^{-1} \right) \gamma (\mu)g
   \end{split}
  \end{equation}
  for all $g\in\ran\Gamma_0$.  
  In particular, for every $g \in \ran \Gamma_0$ the function $\lambda \mapsto M (\lambda) g$ is holomorphic on $\rho 
(A_0)$ and each isolated singularity of  $\lambda \mapsto M (\lambda) g$ is a pole of first order. Moreover, 
$\lim_{\eta \searrow 0} i \eta M (\zeta + i \eta) g$ 
exists for all $g \in \ran \Gamma_0$ and all $\zeta \in \R$.
\item The identity 
\begin{align*}
  \gamma (\mu)^* (A_0 - \lambda)^{-1} \gamma (\nu)g = \frac{M (\lambda)g}{(\lambda - \nu) (\lambda 
- \overline \mu)} + \frac{M (\overline \mu)g}{(\lambda-\overline \mu) 
(\nu-\overline \mu )} + \frac{M (\nu)g}{(\nu - \lambda) ( \nu - \overline \mu)}
\end{align*}
holds for all $g\in\ran\Gamma_0$ if $\lambda\not=\nu$, $\lambda\not=\overline\mu$ and $\nu\not=\overline\mu$.
\end{enumerate}
\end{lemma}

\subsection{Simple symmetric operators and local simplicity}\label{simplesec}

Let $S$ be a closed, densely defined, symmetric operator in the separable Hilbert space
$\cH$. Recall that $S$ is said to be {\em simple} or {\em completely non-selfadjoint} if there is no nontrivial 
$S$-invariant subspace $\cH_0$ of $\cH$ which reduces $S$ to a selfadjoint operator in $\cH_0$, see~\cite[Chapter~VII-81]{AG93}. According to \cite{K49} the simplicity of $S$ is 
equivalent to the density of the span of the defect spaces of $S$ in $\cH$, i.e., $S$ is simple if and only if
\begin{align}\label{eq:simple}
  \cH = \clsp \bigl\{ \ker (S^* - \nu) : \nu \in \C \setminus \R \bigr\} 
\end{align}
holds; here $\clsp$ stands for the closed linear 
span. Assume that $\{ \cG, \Gamma_0, \Gamma_1 \}$ is a quasi boundary triple for $\overline T=S^*$ with $A_0 = T\upharpoonright\ker\Gamma_0$. 
Then it follows that $S$ is simple
if and only if \eqref{eq:simple} holds with $\ker(S^*-\nu)$ replaced by $\ker(T-\nu)$. Moreover, if $\gamma$ is 
the $\gamma$-field corresponding to the quasi boundary triple $\{ \cG, \Gamma_0, \Gamma_1 \}$ we conclude that $S$ is simple if and only if
\begin{align}\label{eq:simple2}
  \cH = \clsp 
\bigl\{ \gamma (\nu) g : \nu \in \C \setminus \R,\, g \in \ran \Gamma_0 \bigr\}
\end{align}
holds. We also mention that the set $\C \setminus \R$ in \eqref{eq:simple2} can be replaced by any set $G \subset \rho (A_0)$ which has an accumulation point 
in each connected component of $\rho (A_0)$; cf. Lemma~\ref{simplelemma}~(v) below.  

Our aim is to generalize the notion of simplicity and to replace it by some weaker, local condition, which is satisfied in, e.g., the applications in Section~\ref{sec:4}. Instead of \eqref{eq:simple2} we will assume that
\begin{align}\label{eq:localSimple'}
  E (\Delta) \cH=\clsp \bigl\{ E (\Delta) \gamma (\nu) g : \nu \in \C \setminus \R,\, g 
\in \ran \Gamma_0 \bigr\}
\end{align}
holds on a Borel set (later on usually an open interval) $\Delta$; here $E(\cdot)$ denotes the spectral measure of~$A_0$. This condition will be imposed in many of the general results in Section~\ref{sec:abstr}. In the next lemma
we discuss this condition and some consequences of it.

\begin{lemma}\label{simplelemma}
Let $S$ be a closed, densely defined, symmetric operator in 
$\cH$ and let $\{ \cG, \Gamma_0, \Gamma_1 \}$ be a quasi boundary triple for $\overline T=S^*$ with $A_0=T\upharpoonright\ker\Gamma_0$. 
Then the following holds.
\begin{enumerate}
  \item If $S$ is simple then \eqref{eq:localSimple'} is satisfied for every Borel set $\Delta\subset\dR$.
  \item If \eqref{eq:localSimple'} holds for 
some Borel set $\Delta\subset\dR$ then 
\begin{align}\label{eq:localSimple2}
  E (\Delta^\prime) \cH=\clsp \bigl\{ E (\Delta^\prime) \gamma (\nu) g : \nu \in \C \setminus \R,\, g 
\in \ran \Gamma_0 \bigr\}
\end{align}
holds for every Borel set $\Delta^\prime\subset\Delta$. 
 \item If $\delta_1, \delta_2, \dots$ are disjoint open intervals such that 
 \begin{align}\label{eq:zerlegungSimple}
  E (\delta_j) \cH=\clsp \bigl\{ E (\delta_j) \gamma (\nu) g : \nu \in \C \setminus \R,\, g 
 \in \ran \Gamma_0 \bigr\} \quad \text{for~all}~j
 \end{align}
 then~\eqref{eq:localSimple'} holds for $\Delta = \bigcup_{j} \delta_j$.
 \item If \eqref{eq:localSimple'} holds for some Borel set $\Delta\subset\dR$ then $\Delta\cap\sigma_{\rm p} (S)=\emptyset$.
 \item If \eqref{eq:localSimple'} holds and $G$ is a subset of $\rho (A_0)$ which has an accumulation point 
in each connected component of $\rho (A_0)$ then 
\begin{align}\label{eq:localSimple99}
  E (\Delta) \cH = \clsp \bigl\{ E (\Delta) \gamma (\nu) g : \nu \in G,\, g \in \ran \Gamma_0 \bigr\}.
\end{align}
\end{enumerate}
\end{lemma}

\begin{proof}
Assertion (i) is a consequence of item (ii) since \eqref{eq:localSimple'} holds with $\Delta = \dR$ when $S$ is simple. 

For~(ii) note that the inclusion $\supset$ in~\eqref{eq:localSimple2}  
clearly holds. For the converse inclusion let $u\in E(\Delta^\prime)\cH$. As $\Delta^\prime\subset\Delta$
we have $u\in E(\Delta)\cH$ and hence there exists a sequence $(v_n)$, $n=1,2,\dots$, in the linear span of 
$\{ E (\Delta) \gamma (\nu) g : \nu \in \C \setminus \R,\, g 
\in \ran \Gamma_0 \}$ which converges to $u$. Then $(E(\Delta^\prime)v_n)$, $n=1,2,\dots$, is a sequence in the  
linear span of 
$\{ E (\Delta^\prime) \gamma (\nu) g : \nu \in \C \setminus \R,\, g 
\in \ran \Gamma_0 \}$ which converges to $E(\Delta^\prime)u=u$.

In order to prove (iii) let $\delta_j$ be as in the assumptions and let $\Delta = \bigcup_{j} \delta_j$. The inclusion $\supset$ in~\eqref{eq:localSimple'} 
again is obvious. For the converse inclusion let $u\in E(\Delta)\cH$ and define
\begin{equation}\label{htildemalwieder}
 \widetilde \cH := \clsp \bigl\{ E (\Delta) \gamma (\nu) g : \nu \in \C \setminus \R,\, g \in \ran \Gamma_0 \bigr\}.
\end{equation}
Since
\begin{equation*}
u=E(\Delta) u=\sum_{j} E(\delta_j)u
\end{equation*}
it is sufficient to show $E(\delta_j)u\in\widetilde \cH$ for all $j$. Note first that by assumption \eqref{eq:zerlegungSimple} we have
\begin{equation*}
 E(\delta_j)u \in \clsp \bigl\{ E (\delta_j) \gamma (\mu) h : \mu \in \C \setminus \R,\, h 
 \in \ran \Gamma_0 \bigr\}
\end{equation*}
and hence the assertion follows if we verify 
\begin{equation}\label{schoenwaers}
 E (\delta_j) \gamma (\mu) h \in \widetilde \cH
\end{equation}
for all $\mu \in \C \setminus \R$, $h \in \ran \Gamma_0$, and all $j$. 
For this purpose consider some fixed $ E (\delta_j) \gamma (\mu) h$. According to Lemma~\ref{lem:gammaWeylProp}~(ii) 
we have 
\begin{align*}
 \gamma (\nu) g = \gamma (\mu) g + (\nu - \mu) (A_0 - \nu)^{-1} \gamma (\mu) g
\end{align*} 
for all $\nu \in \C \setminus \R$ and all $g \in \ran \Gamma_0$, and 
hence $\widetilde\cH$ in \eqref{htildemalwieder} can be rewritten in the form
\begin{align*}
 \widetilde \cH = \clsp \bigg\{ E (\Delta) \gamma (\mu) g, E (\Delta) (A_0 - \nu)^{-1} \gamma (\mu) g : \nu \in \C \setminus \R, g \in \ran \Gamma_0 \bigg\}.
\end{align*}
It follows that for $\eta, \eps > 0$ the element
\begin{align*}
 \int_{\alpha_j + \eta}^{\beta_j - \eta} E (\Delta) \big( ( A_0 - (\lambda + i \eps) )^{-1} - (A_0 - (\lambda - i \eps) )^{-1} \big) \gamma (\mu) h \,d \lambda
\end{align*}
belongs to $\widetilde \cH$, where we have written $\delta_j = (\alpha_j, \beta_j)$. From this and Stone's formula it follows
\begin{align*}
 E (\delta_j) \gamma (\mu) h = E (\delta_j) E (\Delta) \gamma (\mu) h \in \widetilde \cH,
\end{align*}
which proves \eqref{schoenwaers} and, hence, yields the inclusion $\subset$ in \eqref{eq:localSimple'}. Item (iii) is proved.

In order to verify (iv), assume that $S u = \lambda u$ for some $u\in\dom S$ and 
$\lambda \in \Delta$. Then $A_0 u = \lambda u$ and hence $u \in E (\Delta) \cH$. 
On the other hand, for  $g \in \ran \Gamma_0$ and $\nu \in 
\C \setminus \R$ it follows together with Lemma~\ref{lem:gammaWeylProp}~(i) that
\begin{align*}
  (u, E (\Delta) \gamma (\nu) g) = (\gamma (\nu)^* u, g) = \big( 
\Gamma_1 (A_0 - \overline \nu)^{-1} u, g \big) = (\lambda - \overline 
\nu)^{-1} (\Gamma_1 u, g) = 0,
\end{align*}
as $u \in \dom S \subset \ker \Gamma_1$. Hence, $u\in E (\Delta) \cH$ is orthogonal to the linear span of the elements 
$E(\Delta)\gamma(\nu)g$, $\nu\in\C \setminus \R$, $g\in\ran\Gamma_0$, which is dense in $E (\Delta) \cH$ by \eqref{eq:localSimple'}.
This implies $u = 0$ and thus $S$ does not possess eigenvalues in $\Delta$. 

It remains to show (v). 
The inclusion $\supset$ in \eqref{eq:localSimple99} is obvious. In order to prove the inclusion $\subset$
it suffices to verify that the vectors 
$E (\Delta) \gamma (\nu) g$, $g \in \ran \Gamma_0$, $\nu \in G$, span a dense set in $E (\Delta)\cH$. Suppose that $E (\Delta) u$ is orthogonal to this set, that is,
\begin{equation}\label{samstag2}
  0 = (E (\Delta) \gamma (\nu) g,E (\Delta) u)
\end{equation}
holds for all $g \in \ran \Gamma_0$ and all $\nu \in G$. Since $\rho(A_0)\ni\nu\mapsto \gamma 
(\nu) g$ is analytic for each $g \in \ran \Gamma_0$ (see Lemma~\ref{lem:gammaWeylProp}~(ii)) it follows that for each $g \in \ran \Gamma_0$ 
the function $\nu\mapsto (E (\Delta) \gamma (\nu) g, E (\Delta) u)$ is analytic on $\rho(A_0)$, and hence \eqref{samstag2} implies that this function is identically equal to zero. Now \eqref{eq:localSimple'} yields $E (\Delta) u=0$ and (v) follows. 
\end{proof}

\section{Spectral properties of selfadjoint operators and corresponding 
Weyl functions}\label{sec:abstr}

This section contains the main abstract results of this paper. We describe the spectral properties of a given 
selfadjoint operator  by means of a 
corresponding Weyl function. For this we fix the following setting.

\begin{assumption}\label{ass} 
Let $S$ be a closed, densely defined, symmetric operator in the separable Hilbert 
space $\cH$ and let $\{ \cG, \Gamma_0, \Gamma_1\}$ be a quasi boundary 
triple for $\overline T = S^*$ with corresponding $\gamma$-field $\gamma$ and Weyl function $M$. Moreover, let $A_0 = T \upharpoonright \ker \Gamma_0$ and denote by $E (\cdot)$ the spectral measure of $A_0$.
\end{assumption}

\subsection{Eigenvalues and corresponding eigenspaces}

Let us start with a characterization of the isolated and embedded 
eigenvalues as well as the corresponding eigenspaces of a selfadjoint 
operator by means of an associated Weyl function. We write s-$\lim$ for 
the strong limit of an operator function. 

\begin{theorem}\label{thm:eigenGeneral}
Let Assumption~\ref{ass} be satisfied. Then 
$\lambda\in\dR$ is an eigenvalue of $A_0$ such that $\cK := \ker (A_0 - \lambda) 
\ominus \ker (S - \lambda) \neq \{0\}$ if and only if $R_\lambda M :=$ 
\textup{s}-$\lim_{\eta \searrow 0} i \eta M (\lambda + i \eta) \neq 0$. If 
$\dim \cK < \infty$ then the mapping
  \begin{align}\label{eq:tau}
   \tau : \cK \to \ran R_\lambda M, \quad u \mapsto \Gamma_1 u,
  \end{align}
  is bijective; if $\dim \cK = \infty$ then the mapping 
  \begin{align}\label{eq:taugen} 
   \tau : \cK \to \cl_\tau \bigl(\ran R_\lambda M\bigr), \quad u \mapsto \Gamma_1 u,
  \end{align}
  is bijective, where $\cl_\tau$ denotes the closure in the normed space 
$\ran \tau$.
\end{theorem}

\begin{remark}\label{rem:residue}
Recall that the limit  $(R_\lambda M) g = \lim_{\eta \searrow 0} i \eta M (\lambda + i \eta)g$ exists for all 
$\lambda\in\dR$ and all $g\in\ran\Gamma_0$ by Lemma~\ref{lem:gammaWeylProp}~(iv). 
Moreover, if $\lambda$ is an isolated singularity of $M$, that is, 
there exists an open neighborhood $\cO$ of $\lambda$ such that $M$ is strongly holomorphic on $\cO \setminus \{ \lambda \}$, then 
$R_\lambda M \neq 0$ if and only if for some $g \in \ran \Gamma_0$ the 
$\cG$-valued function $\zeta\mapsto M(\zeta) g$ has a pole at $\lambda$. In this 
case $R_\lambda M$ coincides with the residue $\Res_\lambda M$ of $M$ at 
$\lambda$ in the strong sense, i.e.,
\begin{align*}
  (R_\lambda M) g = (\Res_\lambda M)g= \frac{1}{2 \pi i} \int_\cC M (z) g \,d z, \quad g \in 
\ran \Gamma_0,
\end{align*}
where $\cC$ denotes the boundary of an open ball $B$ such that $M$ is strongly holomorphic in a neighborhood of $\overline B$ 
except the point~$\lambda$.
We also remark that without additional assumptions the Weyl function is not able to 
distinguish between isolated and embedded eigenvalues of $A_0$; 
cf.~Proposition~\ref{prop:isolatedEV} below.
\end{remark}

\begin{proof}[Proof of Theorem~\ref{thm:eigenGeneral}]
Let $\lambda \in \R$ be fixed. Note first that the mapping $\Gamma_1 
\upharpoonright \cK$ is injective. Indeed, for $u \in \cK =\ker (A_0 - \lambda) 
\ominus \ker (S - \lambda)$ with 
$\Gamma_1 u = 0$ we have $u \in \ker\Gamma_0\cap\ker\Gamma_1=\dom S$ and $S u = \lambda u$; hence $u 
= 0$. It is our aim to prove the inclusions
\begin{align}\label{eq:incFlambda}
  \ran R_\lambda M \subset \ran (\Gamma_1 \upharpoonright \cK ) \subset 
\overline{\ran R_\lambda M}.
\end{align}
 From this it follows immediately that the mapping $\tau$ 
in~\eqref{eq:tau} and \eqref{eq:taugen} is well-defined and bijective.

In order to verify~\eqref{eq:incFlambda} let $g \in \ran \Gamma_0$ and 
denote by $E(\cdot)$ the spectral measure of $A_0$. Then
\begin{align}\label{calc}
  \big\| i \eta (A_0 - & (\lambda + i \eta))^{-1} \gamma (\nu) g + E 
(\{\lambda\}) \gamma (\nu) g \big\|^2 \nonumber \\
  & = \int_\R \left| \frac{i \eta}{t - (\lambda + i \eta)} + 
\chi_{\{\lambda\}}(t) \right|^2 d (E (t) \gamma (\nu) g, \gamma (\nu) g ) 
\to 0 \quad \text{as} \quad \eta \searrow 0
\end{align}
holds for all $\nu \in \C \setminus \R$. Since by Lemma~\ref{lem:gammaWeylProp}~(i) the 
operator $\gamma (\nu)^*$ is bounded, it follows from~\eqref{calc}
\begin{align}\label{eq:Plambda}
  \lim_{\eta \searrow 0} i \eta \gamma (\nu)^* (A_0 - (\lambda + i 
\eta))^{-1} \gamma (\nu) g = - \gamma (\nu)^* E (\{\lambda\}) \gamma (\nu) g
\end{align}
for all $\nu \in \C \setminus \R$, and together with Lemma~\ref{lem:gammaWeylProp}~(v) we conclude that 
the limit on the left hand side of \eqref{eq:Plambda}
coincides with
\begin{equation}\label{mitternacht}
\lim_{\eta \searrow 0} i \eta
\, \frac{M (\lambda + i \eta)g}{((\lambda + i \eta) - \nu) ((\lambda + i \eta)
- \overline \nu)}=\frac{(R_\lambda M) g}{(\lambda-\nu)(\lambda-\overline\nu)}.
\end{equation}
With the help of 
Lemma~\ref{lem:gammaWeylProp}~(i), \eqref{eq:Plambda} and \eqref{mitternacht} we obtain
\begin{align*}
  \Gamma_1 & E (\{\lambda\}) \gamma (\nu) g \nonumber \\
  & = \Gamma_1 (A_0 - \overline \nu)^{-1} (A_0 - \overline \nu) E 
(\{\lambda\}) \gamma (\nu) g = (\lambda - \overline \nu) \gamma 
(\nu)^* E (\{\lambda\}) \gamma (\nu) g \nonumber \\
  & = - (\lambda - \overline \nu) \lim_{\eta \searrow 0} i \eta \gamma 
(\nu)^* (A_0 - (\lambda + i \eta))^{-1} \gamma (\nu) g = 
\frac{1}{\nu - \lambda}(R_\lambda M) g
\end{align*}
for all $\nu \in \C \setminus \R$. 
Denoting by $P$ the orthogonal projection in $\cH$ onto $\cK = \ker (A_0 
- \lambda) \ominus \ker (S - \lambda)$ it follows
\begin{align}\label{eq:stronglimit}
  \Gamma_1 P \gamma (\nu) g = \frac{1}{\nu - \lambda}(R_\lambda M) g,
\end{align}
where we have used $\Gamma_1 (\ker (S - \lambda)) = \{0\}$. From this 
the first inclusion in~\eqref{eq:incFlambda} follows immediately. 

For 
the second inclusion in~\eqref{eq:incFlambda} note that the mapping $\Gamma_1\upharpoonright\cK$ is continuous as
$\Gamma_ 1u=\gamma(\mu)^*(A_0-\overline\mu)u=(\lambda-\overline\mu)\gamma(\mu)^*u$ holds for all $u\in\cK$ by Lemma~\ref{lem:gammaWeylProp}~(i).
Moreover,
for each $\nu \in \C \setminus \R$ 
the linear space $\{P \gamma (\nu) g : g \in 
\ran \Gamma_0 \}$ is dense in $\cK$. In fact, fix $\nu \in \C \setminus \R$ and let $u 
\in \cK$ be orthogonal to $P \gamma (\nu) g$ for all $g \in \ran 
\Gamma_0$. Then
\begin{align*}
  0 = (u, P \gamma (\nu) g) = (\gamma (\nu)^* u, g) = (\Gamma_1 (A_0 - 
\overline \nu)^{-1} u, g) = (\lambda - \overline \nu)^{-1} (\Gamma_1 u, g)
\end{align*}
by Lemma~\ref{lem:gammaWeylProp}~(i), which implies $\Gamma_1 u = 0$ as $\ran\Gamma_0$ is dense. 
Hence we have $u \in \ker\Gamma_0\cap\ker\Gamma_1=\dom S$ and this implies $u \in \cK\cap \ker (S - \lambda)$, so that $u=0$.
Now the second inclusion 
in~\eqref{eq:incFlambda} follows together with \eqref{eq:stronglimit} and the fact that $\Gamma_1\upharpoonright\cK$ is continuous. 
Hence the mapping $\tau$ 
in~\eqref{eq:taugen} is well-defined and bijective. If $\cK$ is 
finite-dimensional then clearly the closure in~\eqref{eq:taugen} can be 
omitted and we end up with the bijectivity of~\eqref{eq:tau}.
\end{proof}

As an immediate consequence of Theorem~\ref{thm:eigenGeneral} all eigenvalues of $A_0$ 
which are not eigenvalues of $S$ can be characterized as ``generalized 
poles'' of the Weyl function.

\begin{corollary}\label{thm:eigen}
Let Assumption~\ref{ass} be satisfied, and assume that $\lambda \in \R$
is not an eigenvalue of $S$. Then $\lambda$ is an eigenvalue 
of $A_0$ if and only if $R_\lambda M :=$ \textup{s}-$\lim_{\eta \searrow 
0} i \eta M (\lambda + i \eta) \neq 0$. If the multiplicity of the 
eigenvalue $\lambda$ is finite then the mapping
  \begin{align*}
   \tau : \ker (A_0 - \lambda) \to \ran R_\lambda M, \quad u \mapsto 
\Gamma_1 u,
  \end{align*}
  is bijective; if the multiplicity of the eigenvalue $\lambda$ is 
infinite then the mapping
  \begin{align*}
   \tau : \ker (A_0 - \lambda) \to \cl_\tau \bigl(\ran R_\lambda M\bigr), \quad u 
\mapsto \Gamma_1 u,
  \end{align*}
  is bijective, where $\cl_\tau$ denotes the closure in the normed space 
$\ran \tau$.
\end{corollary}

\subsection{Continuous, absolutely continuous, and singular continuous 
spectra}

In this subsection we describe the continuous, absolutely continuous, and singular continuous 
spectrum of a selfadjoint operator $A_0$ by means of the limits of an associated Weyl function $M$. Again we fix the setting in Assumption~\ref{ass}.
It is clear that an additional minimality or simplicity condition must be imposed. Usually one assumes that the 
underlying symmetric operator $S$ is simple; cf. \cite{BMN02}. However, for our purposes the
weaker assumption of local simplicity in Section~\ref{simplesec} is more appropriate: in order to characterize the spectrum of $A_0$ in an open interval $\Delta \subset \R$ we assume that 
\begin{align}\label{eq:localSimple}
  E (\Delta) \cH = \clsp \bigl\{ E (\Delta) \gamma (\nu) g : \nu \in \C \setminus \R,\, g 
\in \ran \Gamma_0 \bigr\}.
\end{align}
For instance, in Theorem~\ref{thm:eigenGeneral} it turned out 
that an eigenvalue $\lambda$ of $A_0$ with its full multiplicity can only be detected by the Weyl function if $\lambda \notin \sigma_{\rm p} (S)$. 
This condition corresponds to the identity~\eqref{eq:localSimple} with $\Delta$ replaced by $\{\lambda\}$; cf.~Lemma~\ref{simplelemma}~(iv).

In the next theorem we agree to say that the Weyl function $M$ can be continued analytically to some point $\lambda \in \R$ if 
there exists an open neighborhood $\cO$ of $\lambda$ in~$\C$ such that 
$\zeta\mapsto M (\zeta) g$ can be continued analytically to $\cO$ for all $g \in \ran 
\Gamma_0$. We mention that the proof of (i) is similar to the proof of \cite[Theorem 1.1]{DLS93}.

\begin{theorem}\label{thm:specTotal}
Let Assumption~\ref{ass} be satisfied, and let 
$\Delta \subset \R$ be an open interval such that the condition~\eqref{eq:localSimple} is satisfied. Then 
the following assertions hold for each $\lambda \in \Delta$.
\begin{enumerate}
  \item $\lambda \in \rho (A_0)$ if and only if $M$ can be 
continued analytically into $\lambda$.
  \item $\lambda \in \sigma_{\rm c} (A_0)$ if and only if 
\textup{s}-$\lim_{\eta \searrow 0} i \eta M (\lambda + i \eta) = 0$ and $M$ cannot be continued analytically into $\lambda$.
\end{enumerate}
If $S$ is simple then~{\rm (i)} and~{\rm (ii)} hold for all $\lambda \in \R$.
\end{theorem}

\begin{proof}
(i) Recall first that by Lemma~\ref{lem:gammaWeylProp}~(iv) the function $\lambda\mapsto M (\lambda) 
g$ is analytic on $\rho (A_0)$ for each $g \in \ran \Gamma_0$, which proves the implication $(\Rightarrow)$. In order 
to verify the implication $(\Leftarrow)$ in (i), let us assume that $M$ can be 
continued analytically to some $\lambda \in \Delta$, that is, there 
exists an open neighborhood $\cO$ of $\lambda$ in $\C$ with $\cO \cap \R 
\subset \Delta$ such that $\zeta\mapsto M (\zeta) g$ can be continued analytically to 
$\cO$ for each $g \in \ran \Gamma_0$. Choose $a, b \in \R$ with $\lambda 
\in (a, b)$, $[a, b] \subset \cO$, and $a, b \notin \sigma_{\rm p} (A_0)$. The spectral projection $E ( (a, 
b) )$ of $A_0$ corresponding to the interval $(a, b)$ is given by 
Stone's formula
\begin{align}\label{stone}
  E ( (a, b) ) = \textup{s-}\hspace{-0.8mm}\lim_{\delta \searrow 0} 
\frac{1}{2 \pi i} \int_a^b \left( (A_0 - (t + i \delta) )^{-1} - (A_0 - 
(t - i \delta) )^{-1} \right) d t,
\end{align}
where the integral on the right-hand side is understood in the strong sense. Using the 
identity in Lemma~\ref{lem:gammaWeylProp}~(v) and~\eqref{stone} a straight forward calculation leads to
\begin{equation*}
 \begin{split}
  \left\| E ( (a,b) ) \gamma (\nu) g \right\|^2 & = \bigl(\gamma(\nu)^* E ( (a,b) ) \gamma (\nu) g,  g\bigr) \\  
 & = \lim_{\delta \searrow 0} 
\frac{1}{2 \pi i} \int_a^b \Bigl( \bigl(\gamma(\nu)^*(A_0 - (t + i \delta) )^{-1}\gamma(\nu)g,g\bigr )  \\ 
& \qquad \qquad \qquad \qquad - \bigl(\gamma(\nu)^* (A_0 - 
(t - i \delta) )^{-1}\gamma(\nu)g,g\bigr) \Bigr) d t=0
 \end{split}
\end{equation*}
for all $g \in \ran \Gamma_0$ and all $\nu \in \C \setminus \R$, since $\zeta\mapsto ( M (\zeta) g, g)$ admits an analytic 
continuation into $\cO$ for all $g \in \ran \Gamma_0$. 
Thus the assumption~\eqref{eq:localSimple} and $[a, b] \subset \Delta$ together with Lemma~\ref{simplelemma}~(ii) imply 
$E ( (a, b) ) = 0$. In particular, $\lambda \in \rho (A_0)$.

(ii) According to Lemma~\ref{simplelemma}~(iv) the condition~\eqref{eq:localSimple} implies that $S$ 
does not have eigenvalues in $\Delta$.  Hence item~(ii) follows immediately from item~(i) and Corollary~\ref{thm:eigen}.

If $S$ is simple then by Lemma~\ref{simplelemma}~(i) the assumption \eqref{eq:localSimple} is satisfied for $\Delta = 
\R$. Hence~(i) and~(ii) hold for all $\lambda \in 
\R$.
\end{proof}

Now we return to the characterization of eigenvalues. We formulate a 
sufficient condition under which the Weyl function is able to 
distinguish between isolated and embedded eigenvalues.

\begin{proposition}\label{prop:isolatedEV}
Let Assumption~\ref{ass} be satisfied and let 
$\Delta \subset \R$ be an open interval. Assume that the condition~\eqref{eq:localSimple} is satisfied and 
let $\lambda \in \Delta$. Then all assertions of Corollary~\ref{thm:eigen} hold for $\lambda$. Moreover, $\lambda$ is 
an isolated eigenvalue of $A_0$ if and only if $\lambda$ is a pole in 
the strong sense of $M$. In this case $R_\lambda M$ is the 
residue of $M$ in the strong sense at $\lambda$; 
cf.~Remark~\ref{rem:residue}.
\end{proposition}

\begin{proof}
Let $\lambda \in \R$ and let $\Delta \subset \R$ be an open interval with $\lambda 
\in \Delta$ such that~\eqref{eq:localSimple} holds. Then $\lambda\not\in\sigma_{\rm p} (S)$ by Lemma~\ref{simplelemma}~(iv)
and hence the assertions in Corollary~\ref{thm:eigen} hold for $\lambda$.
Moreover, if $\lambda$ is an isolated eigenvalue of $A_0$ then by 
Lemma~\ref{lem:gammaWeylProp}~(iv) there exists an open neighborhood $\cO$ of 
$\lambda$ such that $\zeta\mapsto M (\zeta) g$ is holomorphic on $\cO \setminus \{ \lambda \}$ for all $g \in \ran \Gamma_0$. 
From Corollary~\ref{thm:eigen} we conclude that there exists $g \in \ran \Gamma_0$ such that
\begin{align}\label{eq:pole}
\lim_{\eta \searrow 0} i \eta M (\lambda + i \eta) g \neq 0.
\end{align}
Hence Lemma~\ref{lem:gammaWeylProp}~(iv) implies that $M$ has a pole of first order in the strong sense at 
$\lambda$. Conversely, if $M$ has a pole (of first order) in the strong sense at 
$\lambda$ then there exists $g \in \ran \Gamma_0$ such that~\eqref{eq:pole} 
holds. According to Lemma~\ref{lem:gammaWeylProp}~(iv) the order of the pole is one and, 
hence,
\begin{align*}
  \lim_{\eta \searrow 0} i \eta M (\lambda + i \eta) g = (\Res_\lambda M) g \neq \{ 0 \}
\end{align*}
for all $g \in 
\ran \Gamma_0$. It follows with the help of Corollary~\ref{thm:eigen} that $\lambda$ is an eigenvalue of $A_0$. Moreover, Theorem~\ref{thm:specTotal}~(i) implies that 
there exists an open neighborhood $\cO$ of $\lambda$ in $\C$ such that $\cO 
\setminus \{\lambda\} \subset \rho (A_0)$. Hence $\lambda$ 
is isolated in the spectrum of $A_0$. This completes the proof.
\end{proof}

Next we discuss the relation of the function $M$ to the absolutely continuous and singular continuous spectrum of $A_0$. In the special case of 
ordinary boundary triples and $\Delta=\dR$ the following results reduce to those in \cite{BMN02}.
For our purposes a localized version and an extension to quasi boundary triples is necessary. The proofs presented here
are somewhat more direct than those in \cite{BMN02}; in particular, the integral
representation of Nevanlinna functions and the corresponding measures are avoided. 

In the following for a finite 
Borel measure $\mu$ on $\R$ we denote the set of
all growth points of $\mu$ by $\supp \mu$, that is,
\begin{align*}
  \supp \mu = \big\{ x \in \R : \mu ( (x - \eps, x + \eps) ) > 
0~\text{for~all}~\eps > 0 \big\}.
\end{align*}
Note that $\supp \mu$ is closed with $\mu (\R \setminus \supp \mu) = 0$ and that $\supp \mu$ is minimal with this property, 
that is, each closed set $S \subset \R$ with $\mu (\R \setminus S) = 0$ satisfies $\supp \mu \subset S$.
Moreover, for a Borel set $\chi \subset \R$ we define the {\em 
absolutely continuous closure} (also called {\em essential closure}) by
\begin{align*}
  \clac (\chi) := \big\{ x \in \R : \left|(x - \eps, x + \eps) \cap \chi 
\right| > 0 ~\text{for~all}~\eps > 0 \big\},
\end{align*}
where $| \cdot |$ denotes the Lebesgue measure, and the {\em continuous closure} by
\begin{align}\label{eq:clc}
  \clc (\chi) := \big\{ x \in \R : (x - \eps, x + \eps) \cap \chi~\text{is~not~countable~for~all}~\eps > 0 \big\}.
\end{align}
Observe that $\clac (\chi)$ and $\clc (\chi)$ both are closed and that $\clac (\chi) \subset \clc (\chi)  \subset \overline\chi$ holds, 
but in general the converse inclusions are not true. In fact, $\clac (\chi)=\emptyset$ if and only if $|\chi|=0$, 
and $\clc (\chi) = \emptyset$ if and only if $\chi$ is countable.

The following lemma can partly be found in, e.g., the 
monographs~\cite{D74} or~\cite{T09}.

\begin{lemma}\label{mesLem}
Let $\mu$ be a finite Borel measure on $\R$ and denote by $F$ its Borel 
transform,
\begin{align*}
  F (\lambda) = \int_{\R} \frac{1}{t - \lambda} d \mu (t), \quad \lambda 
\in \C \setminus \R.
\end{align*}
Then the limit $\Imag F (x + i0) = \lim_{y \searrow 0} \Imag F (x + i y)$ 
exists and is finite for Lebesgue almost all $x \in \R$. Let $\mu_{\rm 
ac}$ and $\mu_{\rm s}$ be the absolutely continuous and singular part, 
respectively, of $\mu$ in the Lebesgue decomposition $\mu = \mu_{\rm ac} 
+ \mu_{\rm s}$, and decompose $\mu_{\rm s}$ into the singular continuous 
part $\mu_{\rm sc}$ and the pure point part. Then the following 
assertions hold.
\begin{enumerate}
  \item $\supp \mu_{\rm ac} = \clac ( \{ x \in \R : 0 < \Imag F (x + i 
0) < +\infty \} )$.
  \item $\supp \mu_{\rm s} \subset \overline{\{ x \in \R : \Imag F (x + i 0) = + 
\infty\}}$.
  \item $\supp \mu_{\rm sc} \subset \clc (\{ x \in \R : \Imag F (x + i 0) = + 
\infty, \lim_{y \searrow 0} y F ( x + i y) = 0 \})$.
\end{enumerate}
\end{lemma}

\begin{proof}
From~\cite[Lemma~3.14 and Theorem~3.23]{T09} it follows immediately that assertion~(i) is true, that the limit $\Imag F (x + i 0)$ exists and is finite for Lebesgue almost all $x \in \R$, and that
\begin{align}\label{eq:singSupp}
 \mu_{\rm s} \big( \R \setminus \{ x \in \R : \Imag F (x + i 0) = + \infty \} \big) = 0,
\end{align}
which implies~(ii). In order to verify (iii) note first that $\lim_{y \searrow 0} y F (x + i y) = i \mu (\{ x \} )$ holds for all $x \in \R$ since
\begin{align*}
 \big| y F (x + i y) - i \mu ( \{x\} ) \big| \leq \int_{\R} \left| \frac{y}{t - (x + i y)} - i \mathbbm{1}_{\{x\}} (t) \right| d \mu (t) \to 0, \quad y \searrow 0.
\end{align*}
In particular, $\mu ( \{x\} ) \neq 0$ if and only if $\lim_{y \searrow 0} y F (x + i y) \neq 0$. Hence it follows from~\eqref{eq:singSupp} and the definition of $\mu_{\rm sc}$ that
\begin{align}\label{eq:SCmeszero}
 \mu_{\rm sc} ( \R \setminus M_{\rm sc} ) = 0,
\end{align}
where 
\begin{align*}
 M_{\rm sc} := \Bigl\{ x \in \R : \Imag F (x + i 0) = + \infty, \lim_{y \searrow 0} y F (x + i y) = 0 \Bigr\}.
\end{align*}
For $x \in \R \setminus \clc (M_{\rm sc})$ by definition there exists $\eps > 0$ such that $(x - \eps, x + \eps) \cap M_{\rm sc}$ is countable; thus $\mu_{\rm sc} ((x - \eps, x + \eps) \cap M_{\rm sc}) = 0$. With the help of~\eqref{eq:SCmeszero} it follows
\begin{align*}
 \mu_{\rm sc} ( (x - \eps, x + \eps) ) \leq \mu_{\rm sc} ( (x - \eps, x + \eps) \cap M_{\rm sc} ) + \mu_{\rm sc} (\R \setminus M_{\rm sc}) = 0,
\end{align*}
that is, $x \notin \supp \mu_{\rm sc}$.
\end{proof}

The absolutely continuous spectrum of a selfadjoint operator in some interval~$\Delta$ can be 
characterized in the following way.

\begin{theorem}\label{thm:ACtheorem}
Let Assumption~\ref{ass} be satisfied and let 
$\Delta \subset \R$ be an open interval such that the condition
\begin{align}\label{eq:localSimple00}
  E (\delta) \cH = \clsp \bigl\{ E (\delta) \gamma (\nu) g : \nu \in \C \setminus \R,\, g 
\in \ran \Gamma_0 \bigr\}
\end{align}
is satisfied for each open interval $\delta \subset \Delta$ with $\delta \cap \sigma_{\rm p} (S) = \emptyset$. Then the absolutely continuous spectrum of $A_0$ in $\Delta$ is 
given by
  \begin{align}\label{eq:ACidentity}
   \overline{\sigma_{\rm ac} (A_0) \cap \Delta} = \overline{ \bigcup_{g 
\in \ran \Gamma_0} \clac \big( \big\{x \in \Delta : 0 < \Imag (M (x + i 
0) g, g) <
+\infty \big\} \big) }.
  \end{align}
If $S$ is simple then~\eqref{eq:ACidentity} holds for each open interval $\Delta$, including the case $\Delta = \R$. 
\end{theorem}

\begin{proof}
The proof of Theorem~\ref{thm:ACtheorem} consists of two separate steps in which the assertions \eqref{acid1} and \eqref{eq:suppMuAc} below 
will be shown. The identity \eqref{eq:ACidentity} is then an immediate consequence of 
\eqref{acid1} and \eqref{eq:suppMuAc} (note that the right hand side in \eqref{eq:suppMuAc} does not depend on $\zeta\in\C\setminus\R$). We fix some notation first. Let us set
\begin{align}\label{D}
  \cD_\Delta := \bigl\{ E (\Delta) \gamma (\zeta) g : \zeta \in \C \setminus \R ,\, g \in \ran \Gamma_0\bigr\}
\end{align}
and define the measures $\mu_u := (E (\cdot) u, u)$ for $u\in\cH$. Denote by 
$P_{\rm ac}$ the orthogonal projection in $\cH$ onto the absolutely continuous subspace $\cH_{\rm ac}$ of $A_0$.
Observe that
the spectral measure of the absolutely continuous part of $A_0$ is $E(\cdot)P_{\rm ac}$ and that the absolutely continuous 
measures $\mu_{u, \rm ac}$ are given by $\mu_{u, \rm ac}=(E(\cdot)P_{\rm ac}u,P_{\rm ac}u)=\mu_{P_{\rm ac}u}$.

\vskip 0.2cm\noindent
{\bf Step 1.}
In this step the identity
\begin{align}\label{acid1}
   \overline{\sigma_{\rm ac} (A_0) \cap \Delta} = \overline{\bigcup_{u \in \cD_\Delta} \supp \mu_{u, \rm ac}}
\end{align}
will be verified. First of all the open set $\Delta^\prime:=\Delta\backslash\overline{\sigma_{\rm p} (S)}$ is the disjoint union
of open intervals $\delta_j$, $1\leq j\leq N$, $N\in\dN\cup\{\infty\}$, and for each $\delta_j$ we have
\begin{align*}
  E (\delta_j) \cH = \clsp \bigl\{ E (\delta_j) \gamma (\nu) g : \nu \in \C \setminus \R,\, g 
\in \ran \Gamma_0 \bigr\}
\end{align*}
by assumption. With the help of Lemma~\ref{simplelemma}~(iii) we conclude
\begin{align*}
 E (\Delta^\prime) \cH = \clsp \big\{ E (\Delta^\prime) \gamma (\nu) g : \nu \in \C \setminus \R,\, g \in \ran \Gamma_0 \big\}.
\end{align*}
Since $\Delta^\prime\subset\Delta$ it follows immediately that $E (\Delta^\prime) \cH\subset \clsp \cD_\Delta$. Moreover, we have
\begin{align*}
 P_{\rm ac} E (\Delta) \cH = P_{\rm ac} E (\Delta^\prime) \cH \subset  P_{\rm ac}\bigl(\clsp \cD_\Delta\bigr)\subset \clsp P_{\rm ac}\cD_\Delta
 \subset P_{\rm ac} E (\Delta) \cH 
\end{align*}
and therefore
\begin{equation}\label{eq:ohneSchlangeId}
 P_{\rm ac} E (\Delta) \cH=\clsp P_{\rm ac}\cD_\Delta.
\end{equation}

In order to verify~\eqref{acid1}, assume first that $x$ does not belong to the left hand side of \eqref{acid1}, that is, 
$x \notin \overline{\sigma_{\rm ac} (A_0) \cap \Delta}$. Then there exists $\epsilon > 0$ such that
$(x - \epsilon, x + \epsilon) \cap \Delta$ contains no absolutely continuous spectrum of $A_0$. This yields
$$E((x - \epsilon, x + \epsilon) \cap \Delta)P_{\rm ac}=0$$ 
and for $u \in E (\Delta) \cH$ one obtains
\begin{equation*}
 \begin{split}
  \mu_{u, \rm ac} ((x - \epsilon, x + \epsilon)) & = \bigl(E ((x - \epsilon, x + \epsilon))  P_{\rm ac} u,  P_{\rm ac} u\bigr) \\  
  & = \bigl(E ((x - \epsilon, x + \epsilon))  P_{\rm ac} E (\Delta) u,  P_{\rm ac} u\bigr) \\
  & = \bigl(E ((x - \epsilon, x + \epsilon) \cap \Delta) P_{\rm ac} u, P_{\rm ac} u\bigr) \\
  & =0.
\end{split}
\end{equation*}
Therefore $(x - \epsilon, x + \epsilon)\cap \supp \mu_{u, \rm ac}=\emptyset$ for all $u \in E (\Delta) \cH$, in particular, for all $u \in \cD_\Delta$.
Thus
$$
x\not \in\overline{\bigcup_{u \in \cD_\Delta} \supp \mu_{u, \rm ac}} 
$$
and  the inclusion
$\supset$ in~\eqref{acid1} follows. For the converse inclusion assume that $x$ 
does not belong to the right hand side of~\eqref{acid1}. Then there exists $\epsilon > 0$ such that 
$(x - \epsilon, x + \epsilon) \subset \dR \setminus \supp \mu_{u, \rm ac}$ for all $u \in \cD_\Delta$, that is,
\begin{equation*}
 \Vert E ((x - \epsilon, x + \epsilon))  P_{\rm ac} u\Vert^2= \mu_{u, \rm ac} ( (x - \epsilon, x + \epsilon) ) = 0
\end{equation*}
for all $u \in \cD_\Delta$, and hence also for all $u \in \clsp \cD_\Delta$. 
With the help of~\eqref{eq:ohneSchlangeId} it follows 
$$
 E ((x - \epsilon, x + \epsilon)\cap\Delta)  P_{\rm ac} u= E ((x - \epsilon, x + \epsilon))  P_{\rm ac} E(\Delta) u=0
$$
for all $u\in \cH$. This shows that $(x - \epsilon, x + \epsilon) \cap \Delta$ does not contain absolutely continuous spectrum of $A_0$, in particular, 
$x \notin \overline{\sigma_{\rm ac} (A_0) \cap \Delta}$ and the inclusion $\subset$ in~\eqref{acid1} follows.

\vskip 0.2cm\noindent
{\bf Step 2.} In this step we show that the identity
\begin{align}\label{eq:suppMuAc}
\supp \mu_{u, {\rm ac}}  = \clac \bigl( \bigl\{ x \in \Delta : 0 < \Imag 
\bigl(M (x + i 0) g,
g \bigr) < +\infty \bigr\} \bigr)
\end{align}
holds for all $u=E (\Delta) \gamma (\zeta) g \in \cD_\Delta$. Indeed, with the help of the
formula~\eqref{eq:Mformula} we compute
\begin{align}\label{kette}
  \Imag (M & (x + i y) g, g) \nonumber \\
  & = y \| \gamma (\zeta) g \|^2 + \left( |x - \zeta|^2 - y^2 \right) 
\Imag \left( (A_0 - (x + i y) )^{-1} \gamma (\zeta) g, \gamma (\zeta) g 
\right) \nonumber \\
  & \quad  + 2 (x - \Real \zeta) y \Real \left( (A_0 - (x + i y) )^{-1} 
\gamma (\zeta) g, \gamma (\zeta) g \right),
\end{align}
for all $x \in \R$, $y>0$, $g \in \ran \Gamma_0$, and $\zeta \in \C 
\setminus \R$. Moreover, dominated convergence implies that
\begin{align*}
  y \Real \left( (A_0 - (x + i y) )^{-1} \gamma (\zeta) g, \gamma 
(\zeta) g \right) = \int_\R \frac{y (t - x)}{(t - x)^2 + y^2}
d ( E (t) \gamma (\zeta) g, \gamma (\zeta) g )
\end{align*}
converges to zero as $y \searrow 0$. Therefore for $x \in \R$~\eqref{kette} implies
\begin{align}\label{MResId}
  \Imag ( M (x + i 0) g, g) & = |x - \zeta|^2 \Imag \left( (A_0 - (x + i 
0) )^{-1} \gamma (\zeta) g, \gamma (\zeta) g \right),
\end{align}
in the sense that one of the limits exists if and only if the other limit exists, where $+\infty$ is allowed as 
(improper) limit. 

For $u \in \cH$, $x \in \R$, and $y > 0$ the imaginary part of the Borel transform $F_u$ of
the measure $\mu_{u}=(E(\cdot)u,u)$ is given by
\begin{align}\label{eq:Fu}
  \Imag F_{u} (x + i y) & =  \Imag \int_{\R} \frac{1}{t - (x+iy)} d (E(t)u,u) = \Imag\left( (A_0 - (x + i y) )^{-1} u, u \right),
\end{align}
and for $u \in E (\Delta)\cH$ we obtain 
\begin{align*}
 \Imag F_u (x + i 0) = 
 \begin{cases} \Imag \left( (A_0 - (x + i 
0) )^{-1} u, u \right) & \text{if}\,\,\, x \in \Delta, \\
  0 & \text{if}\,\,\, x \notin \overline \Delta,
 \end{cases} 
\end{align*}
in particular, if $u = E (\Delta) \gamma (\zeta) g \in \cD_\Delta$ then 
\begin{align*}
 \Imag F_u (x + i 0) = 
 \begin{cases} \Imag \left( (A_0 - (x + i 
0) )^{-1} \gamma (\zeta) g, \gamma (\zeta) g \right) & \text{if}\,\,\, x \in \Delta, \\
  0 & \text{if}\,\,\, x \notin \overline \Delta.
 \end{cases} 
\end{align*}
Taking into account \eqref{MResId} we then find  
\begin{align}\label{eq:wichtig4}
 \Imag F_u (x + i 0) = 
 \begin{cases} |x - \zeta|^{-2} \Imag (M (x + i 0) g, g) & \text{if}\,\,\, x \in \Delta, \\
  0 & \text{if}\,\,\, x \notin \overline \Delta,
 \end{cases} 
\end{align}
for $u = E (\Delta) \gamma (\zeta) g \in \cD_\Delta$. From Lemma~\ref{mesLem}~(i) we conclude together with \eqref{eq:wichtig4} that
\begin{align*}
\supp \mu_{u, {\rm ac}} & = \clac \bigl( \bigl\{ x \in \Delta : 0 < \Imag F_u (x + i 0) < +\infty \bigr\} \bigr) \nonumber \\
 & = \clac \bigl( \bigl\{ x \in \Delta : 0 < \Imag (M (x + i 0) g, g) < +\infty \bigr\} \bigr)
\end{align*}
holds for $u = E (\Delta) \gamma (\zeta) g \in \cD_\Delta$, which shows \eqref{eq:suppMuAc}.
\end{proof}

Theorem~\ref{thm:ACtheorem} immediately implies the following two corollaries.

\begin{corollary}\label{thm:ACcomplete2}
Let Assumption~\ref{ass} be satisfied and assume 
that~\eqref{eq:localSimple00} holds for each open interval $\delta \subset \R$ such that $\delta \cap \sigma_{\rm p} (S) = \emptyset$. Then
\begin{align*}
  \sigma_{\rm ac} (A_0) = \overline{ \bigcup_{g \in \ran \Gamma_0} \clac 
\big( \big\{x \in \R : 0 < \Imag (M (x + i 0) g, g) <
+\infty \big\} \big) }.
\end{align*}
\end{corollary}

\begin{corollary}\label{cor:ACDelta}
Let Assumption~\ref{ass} be satisfied and let $\Delta \subset \R$ be an open interval such that the condition~\eqref{eq:localSimple} holds. Then the absolutely continuous spectrum of $A_0$ in $\Delta$ is 
given by
  \begin{align*}
   \overline{\sigma_{\rm ac} (A_0) \cap \Delta} = \overline{ \bigcup_{g 
\in \ran \Gamma_0} \clac \big( \big\{x \in \Delta : 0 < \Imag (M (x + i 
0) g, g) <
+\infty \big\} \big) }.
  \end{align*} 
\end{corollary}

In the next corollary a necessary and sufficient condition for the absence of absolutely continuous spectrum
is given.

\begin{corollary}\label{cor:ACequiv}
Let Assumption~\ref{ass} be satisfied and let 
$\Delta \subset \R$ be an open interval. Assume that the condition~\eqref{eq:localSimple00} holds for each open interval $\delta \subset \Delta$ with $\delta \cap \sigma_{\rm p} (S) = \emptyset$. Then $\sigma_{\rm ac} (A_0) \cap \Delta = \emptyset$ if and only if $\Imag (M (x + i 0) g, g) = 0$ holds for all $g \in\ran 
\Gamma_0$ and for almost all $x \in \Delta$. 
\end{corollary}

\begin{proof}
We make use of the fact that for $g \in \ran \Gamma_0$
\begin{equation}\label{zerozero}
 \clac \big(\left\{ x \in \Delta : 0 < \Imag ( M (x + i 0) g, g) < + \infty \right\} \big)=\emptyset
\end{equation}
if and only if
\begin{equation}\label{zeroset}
 \big|\left\{ x \in \Delta : 0 < \Imag ( M (x + i 0) g, g) < + \infty \right\} \big|=0.
\end{equation}
Assume first that  $\sigma_{\rm ac} (A_0) \cap \Delta = \emptyset$. Then 
\eqref{eq:ACidentity} yields \eqref{zerozero} for all  $g \in \ran \Gamma_0$, and hence \eqref{zeroset} holds for all $g \in \ran \Gamma_0$.
Moreover, for $u = \gamma (\zeta) g$ and $x \in \R$ 
by~\eqref{MResId} and~\eqref{eq:Fu} we have
\begin{align*}
  \Imag (M (x + i 0) g, g)  = |x - \zeta|^2 \Imag F_{u} (x + i 0),
\end{align*}
and by Lemma~\ref{mesLem} this limit exists and is finite for 
Lebesgue almost all $x \in \R$. Hence~\eqref{zeroset} implies $\Imag 
(M (x + i 0) g, g) = 0$ for all $g \in \ran \Gamma_0$ and almost all $x 
\in \Delta$. For the converse implication assume that $\Imag (M (x + i 0) g, g) = 0$ for all $g \in\ran 
\Gamma_0$ and for almost all $x \in \Delta$. Then \eqref{zeroset} and hence also \eqref{zerozero} holds for all $g \in\ran 
\Gamma_0$.
Thus \eqref{eq:ACidentity} yields $\sigma_{\rm ac} (A_0) \cap \Delta = \emptyset$. 
\end{proof}

Let us prove next inclusions for the singular and singular continuous spectra of~$A_0$. Recall the definition of the continuous closure $\clc (\chi)$ of a Borel set $\chi$ in~\eqref{eq:clc}.

\begin{theorem}\label{thm:SCtheorem}
Let Assumption~\ref{ass} be satisfied, and let 
$\Delta \subset \R$ be an open interval. Then 
the following assertions hold.
\begin{enumerate}
 \item If the condition~\eqref{eq:localSimple} holds then the singular spectrum of $A_0$ in $\Delta$ satisfies
\begin{align*}
 \bigl(\sigma_{\rm s} (A_0) \cap \Delta\bigr)\, \subset \overline{ \bigcup_{g 
\in \ran \Gamma_0} \big\{x \in \Delta : \Imag (M (x + i 
0) g, g) = +\infty \big\} }.
\end{align*}
 \item If the condition~\eqref{eq:localSimple00} is satisfied for each open interval $\delta \subset \Delta$ with $\delta \cap \sigma_{\rm p} (S) = \emptyset$ then the singular continuous spectrum of $A_0$ in $\Delta$, $\sigma_{\rm sc} (A_0) \cap \Delta$, is contained in the set
\begin{align*}
  \overline{ \bigcup_{g \in \ran \Gamma_0} \clc \big(\big\{x \in \Delta : \Imag (M (x + i 0) g, g) = +\infty, \lim_{y \searrow 0} y (M (x + i y) g, g) = 0 \big\} \big) }.
\end{align*}
\end{enumerate}
If $S$ is simple then {\rm (i)} and {\rm (ii)} hold for each open interval $\Delta$, including 
the case $\Delta = \R$. 
\end{theorem}

\begin{proof}
We show the statements (i) and (ii) at once. Let us define
\begin{align*}
 \cD_\Delta := \bigl\{ E (\Delta) \gamma (\zeta) g : \zeta \in \C \setminus \R,\, g \in \ran \Gamma_0 \bigr\}.
\end{align*}
Note first that the same arguments as in Step~1 of the proof of Theorem~\ref{thm:ACtheorem} imply
\begin{align}\label{eq:SCSrepr}
 \overline{\sigma_i (A_0) \cap \Delta} = \overline{\bigcup_{u \in \cD_\Delta} \supp \mu_{u, i}}, \qquad i = \rm s, sc.
\end{align}
In order to apply Lemma~\ref{mesLem}~(ii) and~(iii), respectively, we calculate the limits that appear there. In fact, it follows from \eqref{eq:Mformula} that for each $g \in \ran 
\Gamma_0$ and each $\zeta \in \C \setminus \R$
\begin{equation}\label{eq:limitTransform}
  \lim_{y\searrow 0}\Imag ( M (x + i y) g, g) = |x - \zeta|^2 \lim_{y\searrow 0} \Imag \left( (A_0 - (x + i y) )^{-1} \gamma (\zeta) g, \gamma (\zeta) g \right)
\end{equation}
and
\begin{equation}\label{eq:limitTransform'}
  \lim_{y\searrow 0} y ( M (x + i y) g, g) = |x - \zeta|^2 \lim_{y\searrow 0}  y \left( (A_0 - (x + i y) )^{-1} \gamma (\zeta) g, \gamma (\zeta) g \right)
\end{equation}
hold; cf.~\eqref{MResId} for the first identity and the text below ~\eqref{MResId} for its 
interpretation as a possible improper limit. 
Let $u = E (\Delta) \gamma (\zeta) g \in \cD_\Delta$ and let 
$$
F_u(x+iy)=\int_{\R} \frac{1}{t - (x+iy)} d (E(t)u,u) =\left( (A_0 - (x + i y) )^{-1} u, u \right)
$$ 
be the Borel transform of $\mu_u = (E (\cdot) u, u)$. Then
\begin{align*}
 \Imag F_u (x + i 0) & = \Imag \left( (A_0 - (x + i 0) )^{-1} E (\Delta) \gamma (\zeta) g, E (\Delta) \gamma (\zeta) g \right)
\end{align*}
for all $x \in \R$. From this we conclude with the help of~\eqref{eq:limitTransform} that
\begin{align}\label{eq:wichtig1}
 \Imag F_u (x + i 0) = 
 \begin{cases} |x - \zeta|^{-2} \Imag (M (x + i 0) g, g) & \text{if}\,\,\, x \in \Delta, \\
  0 & \text{if}\,\,\, x \notin \overline \Delta.
 \end{cases} 
\end{align}
Similarly, from~\eqref{eq:limitTransform'} we obtain
\begin{align}\label{eq:wichtig2}
 \lim_{y \searrow 0} y F_u (x + i y) = 
 \begin{cases} |x - \zeta|^{-2} \lim_{y \searrow 0} y (M (x + i y) g, g) &\text{if}\,\,\, x \in \Delta, \\
  0 & \text{if}\,\,\, x \notin \overline \Delta.
 \end{cases}
\end{align}

It follows from~\eqref{eq:wichtig1}, \eqref{eq:wichtig2}, and Lemma~\ref{mesLem} that
\begin{align*}
 \supp \mu_{u, \rm s} \subset \overline{\bigl\{ x \in \Delta: \Imag ( M (x + i 0) g, g) = + \infty \bigr\}}
\end{align*}
and 
\begin{align*}
 \supp \mu_{u, \rm sc} \subset \clc \Big( \Big\{ x \in \Delta : \Imag ( M (x + i 0) g, g) = + \infty, \lim_{y \searrow 0} y (M (x + i y) g, g) = 0 \Big\} \Big)
\end{align*}
for $u = E (\Delta) \gamma (\zeta) g \in \cD_\Delta$. Thus the assertions of the theorem follow from~\eqref{eq:SCSrepr}.
\end{proof}

We formulate two immediate corollaries which concern the singular continuous spectrum.

\begin{corollary}\label{thm:SCcomplete2}
Let Assumption~\ref{ass} be satisfied and assume
that~\eqref{eq:localSimple00} holds for each open interval $\delta \subset \R$ such that $\delta \cap \sigma_{\rm p} (S) = \emptyset$. 
Then the singular continuous spectrum $\sigma_{\rm sc} (A_0)$ of $A_0$ is contained in the set
\begin{align*}
  \overline{ \bigcup_{g \in \ran \Gamma_0} \clc \big(\big\{x \in \R : \Imag (M (x + i 0) g, g) = +\infty, \lim_{y \searrow 0} y (M (x + i y) g, g) = 0 \big\} \big) }.
\end{align*}
\end{corollary}

\begin{corollary}\label{thm:SCcomplete3}
Let Assumption~\ref{ass} be satisfied, let $\Delta \subset \R$ be an open interval, and assume
that the condition~\eqref{eq:localSimple} holds. Then the singular continuous spectrum $\sigma_{\rm sc} (A_0)$ of $A_0$ in $\Delta$, $\sigma_{\rm sc} (A_0) \cap \Delta$, is contained in the set
\begin{align*}
  \overline{ \bigcup_{g \in \ran \Gamma_0} \clc \big(\big\{x \in \Delta : \Imag (M (x + i 0) g, g) = +\infty, \lim_{y \searrow 0} y (M (x + i y) g, g) = 0 \big\} \big) }.
\end{align*}
\end{corollary}

As a further immediate corollary of Theorem~\ref{thm:SCtheorem} we formulate a sufficient criterion for the absence of singular continuous spectrum in terms of the limiting behaviour of the function~$M$. The 
corresponding result for ordinary boundary triples (in the special case $\Delta = \R$) can be found in~\cite{BMN02}. 

\begin{corollary}\label{cor:SCcor}
Let Assumption~\ref{ass} be satisfied and let 
$\Delta \subset \R$ be an open interval such that the condition~\eqref{eq:localSimple00} is satisfied for each open interval $\delta \subset \Delta$ with $\delta \cap \sigma_{\rm p} (S) = \emptyset$. If for 
each $g \in \ran \Gamma_0$ there exist at most countably many $x \in 
\Delta$ such that
  \begin{align*}
   \Imag (M (x + i y) g, g) \to + \infty \quad \text{and} \quad y (M (x 
+ i y) g, g) \to 0 \quad \text{as} \quad y \searrow 0
  \end{align*}
  then $\sigma_{\rm sc} (A_0) \cap \Delta = \emptyset$.  If $S$ is 
simple the assertion holds for each open interval $\Delta$, including 
the case $\Delta = \R$.  
\end{corollary}

As a further corollary of the theorems of this section we provide 
sufficient criteria for the spectrum of the operator $A_0$ to be 
purely absolutely continuous or purely singular continuous, 
respectively, in some set.

\begin{corollary}
Let Assumption~\ref{ass} be satisfied, let 
$\Delta \subset \R$ be an open interval such that the condition~\eqref{eq:localSimple} is satisfied, and assume that
\begin{align}\label{eq:noeig}
  \lim_{y \searrow 0}  y M (x + i y) g = 0
\end{align}
for all $g \in \ran \Gamma_0$ and all $x \in \Delta$.
Then the following assertions hold.
\begin{enumerate}
  \item If for each $g \in \ran \Gamma_0$ there exist at most countably 
many $x \in \Delta$ such that
$\Imag (M (x + i 0) g, g) = + \infty$, then $\sigma (A_0) \cap \Delta = 
\sigma_{\rm ac} (A_0) \cap \Delta$.
  \item If $\Imag ( M (x + i 0) g, g) = 0$ holds for all $g \in \ran 
\Gamma_0$ and almost all $x \in \Delta$, then $\sigma (A_0) \cap \Delta 
= \sigma_{\rm sc} (A_0) \cap \Delta$.
\end{enumerate}
In particular, if $S$ is simple and $\Delta$ is an arbitrary open interval
such that~\eqref{eq:noeig} holds for all $g \in \ran \Gamma_0$ and all 
$x \in \Delta$ then~{\rm (i)} and~{\rm (ii)} are satisfied.
\end{corollary}

\section{Second order elliptic differential operators on $\R^n$}\label{sec:4}

In this section we show how the spectrum of a selfadjoint second order elliptic differential operator on $\R^n$, $n \geq 2$,
can be described with the help of a Titchmarsh--Weyl function acting on an $n-1$-dimensional compact interface $\Sigma$ which splits $\R^n$ into a
bounded domain $\Omega_{\rm i}$ and
an unbounded domain $\Omega_{\rm e}$ with common boundary $\Sigma$.

We consider the differential expression
\begin{align*}
 \cL  = - \sum_{j, k = 1}^n \frac{\partial}{\partial x_j} a_{jk} \frac{\partial}{\partial x_k} +
 \sum_{j = 1}^n \left( a_j \frac{\partial}{\partial x_j} - \frac{\partial}{\partial x_j} \overline{a_j} \right) + a,
\end{align*}
where $a_{jk}, a_j \in C^\infty(\R^n)$ together with their derivatives are bounded and
satisfy $a_{jk} (x) = \overline{a_{kj} (x)}$ for all $x \in \R^n$, $1 \leq j, k \leq n$,
and $a \in L^\infty(\R^n)$ is real valued. Moreover, we assume that $\cL$ is uniformly elliptic on $\R^n$, that is, there exists $E > 0$ with
 \begin{align}\label{eq:elliptic}
 \sum_{j, k = 1}^n a_{jk} (x) \xi_j \xi_k \geq E \sum_{k = 1}^n \xi_k^2, \quad x \in \R^n, \quad \xi = (\xi_1, \dots, \xi_n)^\top \in \R^n.
\end{align}

The selfadjoint operator
associated with $\cL$ in $L^2 (\R^n)$ is given by
\begin{align}\label{eq:SchroedingerOp}
 A_0 u = \cL u, \quad \dom A_0 = H^2 (\R^n),
\end{align}
where $H^2 (\R^n)$ is the usual $L^2$-based Sobolev space of order $2$ on $\R^n$.
In Sections~\ref{41} and \ref{42} two different choices of Titchmarsh--Weyl functions for the differential expression $\cL$,
both acting on the interface $\Sigma$, are studied. 

\subsection{A Weyl function corresponding to a transmission problem}\label{41}

We first consider a Weyl function for the operator $A_0$ which appears in transmission problems in connection with single layer potentials
(see, e.g. \cite[Chapter 6]{M00} and \cite{R09}) and which was also used in \cite{AP04}
to generalize the classical limit point/limit circle analysis from singular Sturm--Liouville
theory to Schr\"{o}dinger operators in $\dR^3$.

Let $\Sigma$ be the boundary of
a bounded $C^\infty$-domain $\Omega_{\rm i} \subset \R^n$ and denote by $\Omega_{\rm e}$ the exterior of $\Sigma$, that is, $\Omega_{\rm e} = \R^n \setminus \overline{\Omega_{\rm i}}$.
In the following we  make use of operators induced by $\cL$ in $L^2 (\Omega_{\rm i})$ and $L^2 (\Omega_{\rm e})$, respectively.
For $j = {\rm i, e}$ we write $\cL_j$ for the restriction of the differential expression $\cL$ to functions on $\Omega_j$.
For functions in $L^2(\Omega_j)$ we use the index $j$ and we write $u=u_{\rm i}\oplus u_{\rm e}$ for $u\in L^2(\dR^n)$. As
$\Sigma$ is smooth, the selfadjoint Dirichlet operator associated with $\cL_j$ in $L^2 (\Omega_j)$ is given by
\begin{equation*}
 A_{{\rm D}, j} u_j = \cL_j u_j, \quad \dom A_{{\rm D}, j} = \left\{ u_j \in H^2 (\Omega_j) : u_j |_{\Sigma} = 0 \right\},\qquad j = {\rm i, e},
\end{equation*}
where $u_j |_\Sigma$ denotes the trace of $u_j$ at $\Sigma=\partial\Omega_j$.
Let $H^s (\Sigma)$ be the Sobolev spaces of orders $s \geq 0$ on $\Sigma$. We recall that for each
$\lambda \in \rho (A_{{\rm D}, j})$ and each $g \in H^{3/2} (\Sigma)$ there exists a unique solution
$u_{\lambda, j} \in H^2 (\Omega_j)$ of the boundary value problem $\cL_j u_j = \lambda u_j$, $u_j |_\Sigma =g$. This implies that for
each $\lambda \in \rho (A_{{\rm D}, j})$
the Dirichlet-to-Neumann map
\begin{align}\label{eq:DNie}
 \Lambda_j (\lambda) : H^{3/2} (\Sigma) \to H^{1/2} (\Sigma), \quad u_{\lambda,j} |_\Sigma \mapsto \frac{\partial u_{\lambda,j}}{\partial \nu_{\cL_j}} \Big|_\Sigma,
\end{align}
is well-defined; here the conormal derivative with respect to $\cL_j$ in the direction of the outer unit normal
$\nu_j = (\nu_{j,1}, \dots, \nu_{j,n})^\top$ at $\Sigma = \partial \Omega_j$ is defined by
\begin{align*}
 \frac{\partial u}{\partial \nu_{\cL_j}} \Big|_{\Sigma} = \sum_{k, l = 1}^n a_{kl} \nu_{j,k} \frac{\partial u}{\partial x_l} \Big|_{\Sigma} +
 \sum_{k = 1}^n \overline{a_k} \nu_{j,k} u |_{\Sigma}.
\end{align*}
Note that the outer unit normals at $\partial \Omega_{\rm i}$ and $\partial \Omega_{\rm e}$ coincide up to a minus sign. The operator $\Lambda_{\rm i} (\lambda) + \Lambda_{\rm e} (\lambda)$ is invertible for all
$\lambda\in\rho(A_0)\cap\rho (A_{\rm D, i}) \cap \rho (A_{\rm D, e})$ and, hence, the operator function
\begin{align}\label{eq:WeylSchreod}
 \lambda\mapsto M (\lambda) = \big( \Lambda_{\rm i} (\lambda) + \Lambda_{\rm e} (\lambda) \big)^{-1}
\end{align}
is well-defined on $\rho(A_0)\cap\rho (A_{\rm D, i}) \cap \rho (A_{\rm D, e})$. We remark that the values $M (\lambda)$ are bounded operators in $L^2 (\Sigma)$
with domain $H^{1/2} (\Sigma)$; cf.~Lemma~\ref{prop:qbtSchreod} below for the details.

The following theorem is the main result of this section. It states that the absolutely continuous spectrum of $A_0$ can be recovered completely from the knowledge of the
function $M$ in~\eqref{eq:WeylSchreod}, while the eigenvalues and corresponding eigenspaces may be only partially
visible for the function $M$. This depends on the choice of the interface $\Sigma$ and the fact that the symmetric operator
\begin{align}\label{eq:SSchroed}
 S u = \cL u, \quad \dom S = \left\{ u \in H^2 (\R^n) : u |_{\Sigma} = 0 \right\},
\end{align}
may have eigenvalues. In particular, in general $S$ is not simple; cf.~Example~\ref{ex:simple} and Example~\ref{ex:simple2} below.

\begin{theorem}\label{thm:eigenSchroed}
Let $A_0$, $\Sigma$, $S$, and $M$ be as above, let $\lambda,\mu\in\dR$ such that $\lambda\not\in\overline{\sigma_{\rm p} (S)}$,
$\mu \not\in\sigma_{\rm p} (S)$,
and let $\Delta \subset \R$ be an open interval.
Then the following assertions hold.
\begin{enumerate}
 \item $\mu \in\sigma_{\rm p} (A_0)$ if and only if $R_\mu M :=$ \textup{s}-$\lim_{\eta \searrow 0} i \eta M (\mu + i \eta) \neq 0$; if the multiplicity of the eigenvalue $\mu$ is finite then the mapping
 \begin{align}\label{eq:tauSchroed}
  \tau : \ker (A_0 - \mu) \to \ran R_\mu M, \quad u \mapsto u |_\Sigma,
 \end{align}
 is bijective; if the multiplicity of the eigenvalue $\mu$ is infinite then the mapping
 \begin{align}\label{eq:taugenSchroed}
  \tau : \ker (A_0 - \mu) \to \cl_\tau \bigl(\ran R_\mu M\bigr), \quad u \mapsto u |_\Sigma,
 \end{align}
 is bijective, where $\cl_\tau$ denotes the closure in the normed space $\ran \tau$.
 \item $\lambda$ is an isolated eigenvalue of $A_0$ if and only if $\lambda$ is a pole in the strong sense of~$M$.
 In this case \eqref{eq:tauSchroed} and~\eqref{eq:taugenSchroed} hold with $\mu=\lambda$ and $R_\lambda M = \Res_\lambda M$.
 \item $\lambda\in\rho(A_0)$ if and only if $M$ can be continued analytically into $\lambda$.
 \item $\lambda\in\sigma_{\rm c} (A_0)$ if and only if \textup{s}-$\lim_{\eta \searrow 0} i \eta M (\lambda + i \eta) = 0$ and $M$ cannot be continued analytically into $\lambda$.
 \item The absolutely continous spectrum $\sigma_{\rm ac} (A_0)$ of $A_0$ in $\Delta$ is given by
 \begin{equation*}
  \qquad\qquad\overline{\sigma_{\rm ac} (A_0)\cap\Delta} = \overline{ \bigcup_{g \in H^{1/2} (\Sigma)} \clac \big( \big\{x \in \Delta : 0 < \Imag (M (x + i 0) g, g) < +\infty \big\} \big) }
 \end{equation*}
 and, in particular, $\sigma_{\rm ac} (A_0) \cap \Delta = \emptyset$ if and only if $\Imag (M (x + i 0) g, g) = 0$ holds for
 all $g \in H^{1/2} (\Sigma)$ and for almost all $x \in \Delta$.
 \item The singular continous spectrum $\sigma_{\rm sc} (A_0)$ of $A_0$ in $\Delta$ is contained in
 \begin{equation*}
  \overline{ \bigcup_{g \in H^{1/2} (\Sigma)} \clc \big(\big\{x \in \Delta : \Imag (M (x + i 0) g, g) = +\infty, \lim_{y \searrow 0} y (M (x + i y) g, g) = 0 \big\} \big) },
 \end{equation*}
 and, in particular,
 if for each $g \in H^{1/2} (\Sigma)$ there exist at most countably many $x \in \Delta$ such that
 $\Imag (M (x + i y) g, g) \to + \infty$ and $y (M (x + i y) g, g) \to 0$ as $ y \searrow 0$
 then $\sigma_{\rm sc} (A_0) \cap \Delta = \emptyset$.
\end{enumerate}
\end{theorem}

The proof of Theorem~\ref{thm:eigenSchroed} makes use of the following two lemmas and is given at the end of this subsection.

\begin{lemma}\label{prop:qbtSchreod}
Let $S$ be defined as in~\eqref{eq:SSchroed} and let
\begin{align}\label{eq:TSchreod}
 T u = \cL u, \quad \dom T = \left\{ u_{\rm i} \oplus u_{\rm e} \in H^2 (\Omega_{\rm i}) \oplus H^2 (\Omega_{\rm e}) : u_{\rm i} |_\Sigma = u_{\rm e} |_\Sigma \right\}.
\end{align}
Then $\{ L^2 (\Sigma), \Gamma_0, \Gamma_1\}$, where
\begin{align*}
 \Gamma_0, \Gamma_1 : \dom T \to L^2 (\Sigma),\quad \Gamma_0 u = \frac{\partial u_{\rm i}}{\partial \nu_{\cL_{\rm i}}}
 \Big|_\Sigma + \frac{\partial u_{\rm e}}{\partial \nu_{\cL_{\rm e}}} \Big|_\Sigma, \quad \Gamma_1 u = u |_\Sigma,
\end{align*}
is a quasi boundary triple for $S^*$ such that $A_0 = T \upharpoonright \ker \Gamma_0 $ and $\ran\Gamma_0= H^{1/2} (\Sigma)$.
For all $\lambda \in\rho(A_0)\cap \rho (A_{\rm D, i}) \cap \rho (A_{\rm D, e})$
the corresponding Weyl function coincides with the function $M$ in \eqref{eq:WeylSchreod}, and $\dom M (\lambda) = H^{1/2} (\Sigma)$.
\end{lemma}

\begin{proof}
The proof is similar to the proof of \cite[Proposition 3.2]{BLL13}. For the convenience of the reader
we provide the details.
In order to show that $\{ L^2 (\Sigma), \Gamma_0, \Gamma_1\}$ is a quasi boundary triple for $S^*$
we verify (i)-(iii) in the assumptions of Proposition~\ref{prop:ratetheorem}. Recall first that by the classical trace theorem the mapping
$$
H^2(\Omega_j)\rightarrow H^{3/2}(\Sigma)\times H^{1/2}(\Sigma),\qquad u_j\mapsto\left\{u_j\vert_\Sigma, \frac{\partial u_j}{\partial \nu_{\cL_j}}
 \Big|_\Sigma\right\},\quad  j = \rm i, e,
$$
is onto. Hence, for given
$\varphi\in H^{1/2}(\Sigma)$ and $\psi\in H^{3/2}(\Sigma)$ there exist $u_j\in H^2(\Omega_j)$ such that
$$
\frac{\partial u_{\rm i}}{\partial \nu_{\cL_{\rm i}}} \Big|_\Sigma = \varphi,\quad \frac{\partial u_{\rm e}}{\partial \nu_{\cL_{\rm e}}} \Big|_\Sigma = 0,\quad\text{and}\quad
u_{\rm i} |_\Sigma =\psi= u_{\rm e} |_\Sigma,
$$
and it follows $u_{\rm i} \oplus u_{\rm e} \in \dom T$, $\Gamma_0 (u_{\rm i} \oplus u_{\rm e}) = \varphi$, and $\Gamma_1 (u_{\rm i} \oplus u_{\rm e}) = \psi$. This implies that $\ran(\Gamma_0,\Gamma_1)^\top=H^{1/2}(\Sigma)\times H^{3/2}(\Sigma)$. In particular, $\ran(\Gamma_0,\Gamma_1)^\top$
is dense in $L^2(\Sigma)\times L^2(\Sigma)$.
Furthermore, $C_0^\infty(\dR^n\backslash\Sigma)$ is a dense subspace of $L^2(\dR^n)$ which is contained in $\ker\Gamma_0\cap\ker\Gamma_1$.
Thus (i) in Proposition~\ref{prop:ratetheorem} holds. Next we verify the identity \eqref{eq:absGreen} for $u=u_{\rm i}\oplus u_{\rm e},
v=v_{\rm i}\oplus v_{\rm e}\in\dom T$.
With the help of Green's identity and $u\vert_\Sigma=u_j\vert_\Sigma$,  $v\vert_\Sigma=v_j\vert_\Sigma$, $j = \rm i,e$, we compute
\begin{equation*}
\begin{split}
 &(Tu,v)-(u,Tv)=(\cL_{\rm e}u_{\rm e},v_{\rm e})-(u_{\rm e},\cL_{\rm e}v_{\rm e})+(\cL_{\rm i}u_{\rm i},v_{\rm i})-(u_{\rm i},\cL_{\rm i}v_{\rm i})\\
 &\qquad =\left(u_{\rm e}\vert_\Sigma,\frac{\partial v_{\rm e}}{\partial \nu_{\cL_{\rm e}}} \Big|_\Sigma\right)-
          \left(\frac{\partial u_{\rm e}}{\partial \nu_{\cL_{\rm e}}} \Big|_\Sigma,v_{\rm e}\vert_\Sigma\right)+
          \left(u_{\rm i}\vert_\Sigma,\frac{\partial v_{\rm i}}{\partial \nu_{\cL_{\rm i}}} \Big|_\Sigma\right)-
          \left(\frac{\partial u_{\rm i}}{\partial \nu_{\cL_{\rm i}}} \Big|_\Sigma,v_{\rm i}\vert_\Sigma\right)\\
 &\qquad =\left(u\vert_\Sigma, \frac{\partial v_{\rm i}}{\partial \nu_{\cL_{\rm i}}} \Big|_\Sigma+\frac{\partial v_{\rm e}}{\partial \nu_{\cL_{\rm e}}} \Big|_\Sigma\right)-
          \left(\frac{\partial u_{\rm i}}{\partial \nu_{\cL_{\rm i}}} \Big|_\Sigma+\frac{\partial u_{\rm e}}{\partial \nu_{\cL_{\rm e}}} \Big|_\Sigma,v\vert_\Sigma\right)\\
 &\qquad=(\Gamma_1 u,\Gamma_0 v)-(\Gamma_0 u,\Gamma_1 v).
 \end{split}
 \end{equation*}
We have shown that (ii) in Proposition~\ref{prop:ratetheorem} holds. Finally it is not difficult to see that $\dom A_0=H^2(\dR^n)$ is contained in
$\ker\Gamma_0$, that is, assumption (iii) in Proposition~\ref{prop:ratetheorem} is satisfied. Therefore we obtain from Proposition~\ref{prop:ratetheorem}
that $T\upharpoonright(\ker\Gamma_0\cap\ker\Gamma_1)$ is a densely defined, closed, symmetric operator in $L^2(\dR^n)$, that $\{ L^2 (\Sigma), \Gamma_0, \Gamma_1\}$
is a quasi boundary triple for its adjoint and that $A_0 = T\upharpoonright\ker\Gamma_0$. In particular, $T\upharpoonright\ker\Gamma_0$ is defined on $H^2(\dR^n)$.
Hence $T\upharpoonright(\ker\Gamma_0\cap\ker\Gamma_1)$ coincides with the symmetric operator $S$ in \eqref{eq:SSchroed}
and $\{ L^2 (\Sigma), \Gamma_0, \Gamma_1\}$ is a quasi boundary triple for $\overline T=S^*$.  It remains to check that the corresponding Weyl function has the form
\eqref{eq:WeylSchreod}. For this let $\lambda \in\rho(A_0)\cap \rho (A_{\rm D, i}) \cap \rho (A_{\rm D, e})$ and let
$u_\lambda=u_{\lambda, \rm i}\oplus u_{\lambda, \rm e} \in \ker (T - \lambda)$, that is, $u_{\lambda, j} \in H^2 (\Omega_j)$, $j = \rm i, e$, $u_{\lambda, \rm i} |_\Sigma = u_{\lambda, \rm e} |_\Sigma$, and $\cL_j u_{\lambda,j}=\lambda u_{\lambda,j}$, $j = \rm i, e$. Then we have
\begin{align}\label{eq:couplingWeyl}
  \bigl(\Lambda_{\rm i} (\lambda) + \Lambda_{\rm e} (\lambda)\bigr)\Gamma_1 u_\lambda
=\frac{\partial u_{\lambda,\rm i}}{\partial \nu_{\cL_{\rm i}}} \Big|_\Sigma+\frac{\partial u_{\lambda,\rm e}}{\partial \nu_{\cL_{\rm e}}} \Big|_\Sigma =\Gamma_0 u_\lambda.
\end{align}
Note further that $\Lambda_{\rm i} (\lambda) + \Lambda_{\rm e} (\lambda)$ is injective for all $\lambda \in\rho(A_0)\cap \rho (A_{\rm D, i}) \cap \rho (A_{\rm D, e})$.
In fact, assume
$\Gamma_1 u_\lambda \in \ker (\Lambda_{\rm i} (\lambda) + \Lambda_{\rm e} (\lambda))$. Then \eqref{eq:couplingWeyl}
implies $u_\lambda \in \ker \Gamma_0 = \dom A_0$, and it follows $u_\lambda \in \ker (A_0 - \lambda)$. Since $\lambda \in \rho (A_0)$
we obtain $u_\lambda = 0$ and, hence, $\Gamma_1 u_\lambda = 0$. Therefore it follows from \eqref{eq:couplingWeyl} that
the Weyl function corresponding to $\{\cG,\Gamma_0,\Gamma_1\}$ coincides with the function $M$ in~\eqref{eq:WeylSchreod}.
\end{proof}

In the next lemma it is shown that $S$ satisfies the local simplicity in the assumptions of the results in Section~\ref{sec:abstr}. 

\begin{lemma}\label{sonnenschein}
Let $A_0$ be the selfadjoint elliptic operator in \eqref{eq:SchroedingerOp} with spectral measure $E (\cdot)$ and let $S$ be the symmetric operator in \eqref{eq:SSchroed}.
Let $\{L^2 (\Sigma), \Gamma_0, \Gamma_1\}$ be the quasi boundary triple in Lemma~\ref{prop:qbtSchreod} and let $\gamma$ be the corresponding $\gamma$-field.
Then
$$ \clsp \bigl\{ E (\delta) \gamma (\nu) g : g \in H^{1/2} (\Sigma),\,\nu \in \C \setminus \R \bigr\} = E (\delta) L^2 (\R^n)$$
holds for every open interval $\delta\subset\dR$ such that $\delta\cap\sigma_{\rm p}(S)=\emptyset$.
\end{lemma}

\begin{proof}
For $j = \rm i, e$ we consider the densely defined, closed, symmetric operators
\begin{align*}
 S_j u_j = \cL_j u_j, \quad \dom S_j = \bigg\{ u_j \in H^2 (\Omega_j) : u_j |_\Sigma = \frac{\partial u_j}{\partial \nu_{\cL_j}} \Big|_\Sigma = 0
 \bigg\},
\end{align*}
in $L^2 (\Omega_j)$ and the operators
\begin{align*}
 T_j u_j = \cL_j u_j, \quad \dom T_j = H^2 (\Omega_j),
\end{align*}
in $L^2 (\Omega_j)$. It is not difficult to verify that $\{L^2(\Sigma),\Gamma_0^j,\Gamma_1^j\}$, where
\begin{align*}
 \Gamma_0^j, \Gamma_1^j : \dom T_j \to L^2 (\Sigma),\quad \Gamma_0^j u_j = u_j |_{\Sigma}, \quad \Gamma_1^j u_j = - \frac{\partial u_j}{\partial \nu_{\cL_j}} \Big|_\Sigma,
\end{align*}
is a quasi boundary triple for $S_j^*$, $j = \rm i, e$; cf. \cite[Proposition 4.1]{BL07}.
For $\lambda \in \rho (A_{\rm D, j})$, $j = \rm i, e$, the corresponding $\gamma$-fields are given by
\begin{align*}
 \gamma_j (\lambda):L^2(\Sigma)\supset H^{3/2}(\Sigma)\rightarrow L^2(\Omega_j),\qquad \varphi \mapsto \gamma_j (\lambda)\varphi= u_{\lambda,j},
\end{align*}
where $u_{\lambda,j}$ is the unique solution in $H^2(\Omega_j)$ of $\cL_j u_j=\lambda u_j$, $u_j\vert_\Sigma=\varphi$. It follows in the same way as in
\cite[Proposition 2.2]{BR13} that $S_{\rm e}$ is simple; the simplicity of $S_{\rm i}$ follows from a unique continuation argument, see, e.g. \cite[Proposition 2.5]{BR12}.
Therefore we have
\begin{align*}
 L^2 (\Omega_j) = \clsp \bigl\{ \gamma_j (\nu) g : g \in H^{3/2} (\Sigma), \, \nu \in \C \setminus \R \bigr\}, \quad j = \rm i, e,
\end{align*}
and hence
\begin{align}\label{eq:gammaDensElliptic2}
\begin{split}
 L^2 (\R^n) &= L^2 (\Omega_{\rm i})\oplus L^2 (\Omega_{\rm e})\\
            &=\clsp \bigl\{ \gamma_{\rm i} (\mu) g\oplus \gamma_{\rm e} (\nu)h : g,h \in H^{3/2} (\Sigma), \, \mu, \nu \in \C \setminus \R \bigr\}.
\end{split}
 \end{align}
Here and in the following $\oplus$ denotes the orthogonality of the closed subspaces $L^2 (\Omega_{\rm i})$ and $L^2 (\Omega_{\rm e})$ in  $L^2 (\R^n)$.

Let now $\delta\subset\dR$ be an open interval such that $\delta\cap\sigma_{\rm p}(S)=\emptyset$ and let $T$ be as in \eqref{eq:TSchreod}. Since
\begin{equation}\label{neuesfeld}
\bigl\{\gamma_{\rm i} (\nu) g \oplus \gamma_{\rm e}(\nu)g: g \in H^{3/2} (\Sigma)\bigr\} = \ker (T - \nu) = \ran \gamma (\nu),\qquad \nu \in \C \setminus \R,
\end{equation}
we have to verify that
\begin{equation*}
\cH_\delta := \clsp \big\{ E (\delta) (\gamma_{\rm i} (\nu) g\oplus \gamma_{\rm e} (\nu)g) : g \in H^{3/2} (\Sigma) , \,\nu \in \C \setminus \R \big\} = E (\delta) L^2 (\R^n).
\end{equation*}
We note first that the inclusion $\cH_\delta\subset E (\delta) L^2 (\R^n)$ is obviously true.
For the opposite inclusion we conclude from \eqref{eq:gammaDensElliptic2} that it suffices to verify
\begin{equation}\label{simplificationSimpleSchroed}
 \begin{split}
  E (\delta) (\gamma_{\rm i} (\mu) g \oplus 0) \in \cH_\delta,&\qquad g \in H^{3/2} (\Sigma), \, \mu \in \C \setminus \R,\\
  E (\delta) (0 \oplus \gamma_{\rm e} (\nu) h ) \in \cH_\delta,&\qquad h \in H^{3/2} (\Sigma), \, \nu \in \C \setminus \R.
 \end{split}
\end{equation}

Let us show the statements in~\eqref{simplificationSimpleSchroed}. We start with the second one.
Let us fix $\mu \in \C \setminus \R$. By Lemma~\ref{lem:gammaWeylProp}~(ii) we have
\begin{align*}
 \gamma_j (\nu) h = \big( I + (\nu - \mu) (A_{{\rm D}, j} - \nu)^{-1} \big) \gamma_j (\mu) h, \quad h \in H^{3/2} (\Sigma),\, \nu \in \C \setminus \R,
\end{align*}
$j = \rm i, e$. From this it follows
\begin{equation*}
 \begin{split}
 \cH_\delta & = \clsp \big\{ E (\delta) (\gamma_{\rm i} (\nu) h\oplus \gamma_{\rm e} (\nu)h) : h \in H^{3/2} (\Sigma) , \,\nu \in \C \setminus \R \big\}\\
 & = \clsp \Big\{ E (\delta) (\gamma_{\rm i} (\mu) h \oplus \gamma_{\rm e} (\mu) h), \\
 & \quad E (\delta) \left((A_{\rm D, i} - \nu)^{-1} \gamma_{\rm i} (\mu) h \oplus  (A_{\rm D, e} - \nu)^{-1} \gamma_{\rm e} (\mu) h \right) : h \in H^{3/2} (\Sigma), \,\nu \in \C \setminus \R \Big\}.
 \end{split}
\end{equation*}
Since $A_{\rm D, i}$ and $A_{\rm D, e}$ are both semibounded from below we may choose $\lambda_0 \in \R$ such that
$\sigma (A_{{\rm D}, j}) \subset (\lambda_0, \infty)$, $j = \rm i, e$. Recall that the spectrum of $A_{\rm D, i}$ is purely discrete and
let $\lambda_1 < \lambda_2 < \dots$ be the distinct eigenvalues of $A_{\rm D, i}$. Then for all $\eta, \eps > 0$ and $k=0, 1, 2,\dots$ the function
\begin{align*}
 E (\delta) \Bigg[\int_{\lambda_k + \eta}^{\lambda_{k + 1} - \eta} & \big( (A_{\rm D, i} - (\lambda + i \eps) )^{-1} - (A_{\rm D, i} - (\lambda - i \eps) )^{-1} \big)
 \gamma_{\rm i} (\mu) h\, d\lambda \\
 &  \oplus \int_{\lambda_k + \eta}^{\lambda_{k + 1} - \eta} \big( (A_{\rm D, e} - (\lambda + i \eps) )^{-1} - (A_{\rm D, e} - (\lambda - i \eps) )^{-1} \big)
 \gamma_{\rm e} (\mu) h \,d \lambda\Bigg]
\end{align*}
belongs to $\cH_\delta$, and as $(\lambda_k,\lambda_{k+1})\subset\rho(A_{\rm D, i})$, Stone's formula implies
\begin{align}\label{eq:defectExpr}
 E (\delta) \bigl(0 \oplus E_{\rm e} ((\lambda_k, \lambda_{k + 1})) \gamma_{\rm e} (\mu) h \bigr) \in \cH_\delta,
\end{align}
where $E_{\rm e} (\cdot)$ is the spectral measure of $A_{\rm D, e}$.
Next we show that for the eigenvalues $\lambda_k$, $k=1,2,\dots,$ of $A_{\rm D, i}$ the property
\begin{align}\label{eq:defectExpr3}
 E (\delta) \bigl(0 \oplus E_{\rm e} (\{\lambda_k\}) \gamma_{\rm e} (\mu) h \bigr) \in \cH_\delta
\end{align}
holds. For this consider the element
$$u = 0 \oplus E_{\rm e} (\{\lambda_k\}) \gamma_{\rm e} (\mu) h$$
for some fixed $h\in H^{3/2} (\Sigma)$. Clearly, as
$u \in \ker ((A_{\rm D, i} \oplus A_{\rm D, e}) - \lambda_k)$ and as $A_{\rm D, i} \oplus A_{\rm D, e}$ is a selfadjoint extension
of the symmetric operator $S$ in \eqref{eq:SSchroed} we may write
$u$ in the form $u=u_{\rm D}\widetilde\oplus u_S$ with $u_S\in\ker(S-\lambda_k)$ and
\begin{equation}\label{ud}
u_{\rm D}\in\ker\bigl((A_{\rm D, i} \oplus A_{\rm D, e}) - \lambda_k\bigr)\widetilde\ominus\ker(S-\lambda_k),
\end{equation}
where $\widetilde\oplus$ and $\widetilde \ominus$ indicate the orthogonality of subspaces in $\ker ((A_{\rm D, i} \oplus A_{\rm D, e}) - \lambda_k)$.
Then
for each $v \in \bigcap_{\nu \in \C \setminus \R} \ran (S - \nu)$ and each $\nu \in \C \setminus \R$ one has
\begin{align}\label{eq:calc}
\begin{split}
 (v,u_{\rm D}) & = ( (S - \nu) (S -  \nu)^{-1} v,u_{\rm D})
 = \bigl((S -  \nu)^{-1} v,((A_{\rm D, i} \oplus A_{\rm D, e}) - \overline\nu) u_{\rm D}\bigr) \\
 & = (\lambda_k - \nu) ((S -  \nu)^{-1} v,u_{\rm D}).
\end{split}
 \end{align}
Since the limit
$$y:=\lim_{\eta \searrow 0} \eta \bigl(S - (\lambda_k + i \eta)\bigr)^{-1} v=
\lim_{\eta \searrow 0} \eta \bigl((A_{\rm D, i} \oplus A_{\rm D, e} )- (\lambda_k + i \eta)\bigr)^{-1} v$$
exists and
\begin{equation*}
\begin{split}
\bigl(y,(S^*-\lambda_k)w \bigr)
 &=\lim_{\eta\searrow 0}\eta\bigl(\bigl(S - (\lambda_k + i \eta)\bigr)^{-1} v,(S^*-\lambda_k)w\bigr)\\
 &=\lim_{\eta\searrow 0}\eta\bigl((S-\lambda_k)\bigl(S - (\lambda_k + i \eta)\bigr)^{-1} v,w\bigr)\\
 &=\lim_{\eta\searrow 0}\eta\bigl[(v,w)+\bigl(i\eta\bigl(S - (\lambda_k + i \eta)\bigr)^{-1} v,w\bigr)\bigr]=0
\end{split}
\end{equation*}
holds for all $w \in \dom S^*$ we conclude that
$$y=\lim_{\eta \searrow 0} \eta \bigl(S - (\lambda_k + i \eta)\bigr)^{-1} v\in \bigl(\ran (S^*-\lambda_k)\bigr)^\bot=\ker(S-\lambda_k).$$
In particular, \eqref{ud} implies $(y,u_{\rm D})=0$.
Therefore we obtain from the identity~\eqref{eq:calc} with
$\nu = \lambda_k + i \eta$ in the limit
\begin{equation*}
 (v,u_{\rm D})=-i\,\lim_{\eta \searrow 0}\eta \bigl(\bigl(S - (\lambda_k + i \eta)\bigr)^{-1} v,u_{\rm D}\bigr)=-i(y,u_{\rm D})=0.
\end{equation*}
This shows that
$u_{\rm D}$ is orthogonal to $\bigcap_{\nu \in \C \setminus \R} \ran (S - \nu)$
and hence
\begin{equation*}
u_{\rm D}  \in \clsp \bigl\{ \ker (S^* - \nu) : \nu \in \C \setminus \R \bigr\}= \clsp \bigl\{ \ker (T - \nu) : \nu \in \C \setminus \R \bigr\}.
\end{equation*}
Therefore \eqref{neuesfeld} implies
\begin{equation}\label{eq:uD}
u_{\rm D}  \in \clsp \bigl\{ \gamma_{\rm i} (\nu) h \oplus \gamma_{\rm e}(\nu)h : h \in H^{3/2} (\Sigma),\, \nu \in \C \setminus \R\bigr\}.
\end{equation}
Note that if the eigenvalue $\lambda_k$ of $A_{\rm D, i}$ is contained in the interval $\delta$ then by assumption $\lambda_k\not\in\sigma_{\rm p}(S)$ and hence
$u=u_{\rm D}$ in this case. If $\lambda_k\not\in\delta$ then $u_S \in \ker (S - \lambda_k) \subset \ker (A_0 - \lambda_k)$ implies that $u_S$ is
orthogonal to $\ran E (\delta)$, so that $E (\delta) u_S=0$.
Summing up we have for any eigenvalue $\lambda_k$, $k=1,2,\dots$, of $A_{\rm D, i}$ that
\begin{equation*}
E (\delta) \bigl(0 \oplus E_{\rm e} (\{\lambda_k\}) \gamma_{\rm e} (\mu) h\bigr) = E ( \delta) u = E ( \delta) (u_S \widetilde\oplus u_{\rm D}) =
E (\delta) u_{\rm D} \in \cH_\delta
\end{equation*}
by~\eqref{eq:uD}. We have shown \eqref{eq:defectExpr3}.

Let $m\in\N$. Then we have
$$
E_{\rm e}((-\infty,\lambda_m))\gamma_{\rm e} (\mu) h= \sum_{k=1}^{m-1} E_{\rm e}(\{\lambda_k\})\gamma_{\rm e} (\mu) h+\sum_{k=0}^{m-1}
E_{\rm e}((\lambda_k,\lambda_{k+1}))\gamma_{\rm e} (\mu) h
$$
and from \eqref{eq:defectExpr} and \eqref{eq:defectExpr3} we conclude
\begin{align*}
 E (\delta) \big( 0 \oplus E_{\rm e} (-\infty, \lambda_m) \gamma_{\rm e} (\mu) h \big) \in \cH_\delta.
\end{align*}
Taking the limit $m \nearrow + \infty$ we obtain $E (\delta) (0 \oplus \gamma_{\rm e} (\mu) h) \in \cH_\delta$. We have proved the second statement in
\eqref{simplificationSimpleSchroed}.

For the first statement in \eqref{simplificationSimpleSchroed} observe that for $\mu\in\C\setminus\R$ fixed, $g\in H^{3/2}(\Sigma)$ and
$k=1,2,\dots$
\begin{align*}
 E (\delta) \bigl(E_{\rm i} (\{\lambda_k\}) \gamma_{\rm i} (\mu) g \oplus 0 \bigr) \in \cH_\delta
\end{align*}
can be verified in the same way as \eqref{eq:defectExpr3}, where $E_{\rm i} (\cdot)$ is the spectral measure of $A_{\rm D, i}$. Hence
for $m\in\N$ we conclude
\begin{align*}
 E (\delta) \bigl(E_{\rm i} ((-\infty,\lambda_m)) \gamma_{\rm i} (\mu) g \oplus 0 \bigr) \in \cH_\delta
\end{align*}
and in the limit $m \nearrow + \infty$ we obtain the first statement in \eqref{simplificationSimpleSchroed}. 

Now \eqref{simplificationSimpleSchroed} together with \eqref{eq:gammaDensElliptic2} imply the inclusion $E(\delta) L^2(\dR^n)\subset \cH_\delta$. This completes 
the proof of Lemma~\ref{sonnenschein}.
\end{proof}

As a consequence of Lemma~\ref{sonnenschein} we obtain the following corollary.

\begin{corollary}\label{simplecor}
The operator $S$ in \eqref{eq:SSchroed} is simple if and only if $\sigma_{\rm p}(S)=\emptyset$.
\end{corollary}

\begin{proof}[{\bf Proof of Theorem~\ref{thm:eigenSchroed}}]
Let $\{L^2 (\Sigma), \Gamma_0, \Gamma_1 \}$ be the quasi boundary triple for $\overline T = S^*$ in Lemma~\ref{prop:qbtSchreod}. Then 
$T \upharpoonright \ker \Gamma_0$ corresponds to the selfadjoint elliptic differential operator $A_0$ in~\eqref{eq:SchroedingerOp} and 
the associated Weyl function coincides with the operator function $M$ in~\eqref{eq:WeylSchreod}. Taking Lemma~\ref{sonnenschein} into account, 
item (i) follows from Corollary~\ref{thm:eigen} and items (ii)-(iv) are 
consequences of Theorem~\ref{thm:specTotal} and Proposition~\ref{prop:isolatedEV} when choosing an open interval $\delta \ni \lambda$ 
with $\delta \cap \sigma_{\rm p} (S) = \emptyset$. Moreover, item~(v) follows from Theorem~\ref{thm:ACtheorem} and Corollary~\ref{cor:ACequiv}, 
and item~(vi) is due to Theorem~\ref{thm:SCtheorem} and Corollary~\ref{cor:SCcor}.
\end{proof}

We point out that in the case that the symmetric operator $S$ is simple the assertions in Theorem~\ref{thm:eigenSchroed} hold for
all $\lambda,\mu\in\dR$. On the other hand, without further assumptions, it may happen that $S$ possesses eigenvalues.
In this case at least the parts of the eigenspaces of $A$ which do not belong to $S$ can be characterized
in terms of the function $M$; cf.~Theorem~\ref{thm:eigenGeneral}. The next examples illustrate
that a proper choice of the interface $\Sigma$ may avoid eigenvalues of $S$.

\begin{example}\label{ex:simple}
Assume that $\cL$ equals the Laplacian outside some compact set $K \subset \R^n$ and choose $\Sigma$ to be the boundary of any smooth, 
bounded domain $\Omega_{\rm i} \supset K$. Then $S$ does not 
have any eigenvalues. Indeed, if $u \in H^2 (\R^n)$ satisfies $\cL u = \lambda u$ on $\R^n$ and $u |_\Sigma = 0$ then
$u |_{\Omega_{\rm e}}$ belongs to $\ker (A_{\rm D, e} - \lambda)$ and must vanish. Then a unique continuation argument implies $u = 0$.
Hence $S$ is simple by Corollary~\ref{simplecor} and the assertions in Theorem~\ref{thm:eigenSchroed} hold for all $\lambda,\mu\in\dR$. 
\end{example}

\begin{example}\label{ex:simple2}
Let the coefficients of $\cL$ be chosen in a way such that for some bounded, smooth domain $\Omega_{\rm i} \subset \R^n$ the operator
$A_{\rm D,i}$ in $L^2 (\Omega_{\rm i})$ is strictly positive; for instance this happens if $- \frac{2}{E} \sum_{j = 0}^n \|a_j\|_\infty^2 + \inf a \geq 0$ 
on $\Omega_{\rm i}$, 
where $E$ is an ellipticity constant for $\cL$, see~\eqref{eq:elliptic}. If we choose $\Sigma = \partial \Omega_{\rm i}$ then $S$ has no non-positive eigenvalues,
otherwise $S u = \lambda u$ for some $\lambda \leq 0$ and $u \in \dom S$ with $u \neq 0$, and a unique
continuation argument yields that $u_{\rm i}$ is nontrivial, thus $u_{\rm i}$ is an eigenfunction of $A_{\rm D, i}$ corresponding to
the eigenvalue $\lambda \leq 0$, a contradiction. Hence in this situation all non-positive eigenvalues of $A_0$
and the corresponding eigenspaces can be described completely in terms of the function $M$.
\end{example}

\subsection{A block operator matrix Weyl function associated with a decoupled system}\label{42}

In this section we consider a different Weyl function for the operator $A_0$, which corresponds to a
symmetric operator which is always simple, independently of the choice of the interface $\Sigma$. This symmetric operator is the orthogonal sum of the minimal symmetric realizations $S_{\rm i}$ and $S_{\rm e}$ of $\cL$ in $L^2 (\Omega_{\rm i})$ and $L^2 (\Omega_{\rm e})$, respectively,
in the proof of Lemma~\ref{sonnenschein}, and hence an infinite dimensional restriction of the symmetric operator in \eqref{eq:SSchroed};
it can be viewed as a decoupled symmetric operator.
Let $\Lambda_{\rm i}$ and $\Lambda_{\rm e}$ be the Dirichlet-to-Neumann maps for the interior and exterior elliptic boundary value problem, 
respectively, defined in~\eqref{eq:DNie}, and let
 \begin{equation*}
 A_{\rm N, e} u_{\rm e} = \cL_{\rm e} u_{\rm e}, \quad \dom A_{{\rm N,e}} = 
 \left\{ u_{\rm e} \in H^2 (\Omega_{\rm e}) : \frac{\partial u_{\rm e}}{\partial \nu_{\cL_{\rm e}}} \Big|_\Sigma= 0 \right\},
\end{equation*}
be the selfadjoint 
realization of $\cL_{\rm e}$ in $L^2(\Omega_{\rm e})$
with Neumann boundary conditions. In Lemma~\ref{prop:qbtSchreod3} below it will turn out that the function
\begin{equation}\label{widetildem}
 \lambda \mapsto \widetilde M(\lambda)= \begin{pmatrix} \Lambda_{\rm i}(\lambda) & 1 \\ 1 & -\Lambda_{\rm e}(\lambda)^{-1}\end{pmatrix}^{-1} \quad \text{in}~L^2 (\Sigma) \times L^2 (\Sigma)
\end{equation}
is well defined on $\rho(A_0)\cap \rho (A_{\rm D, i}) \cap \rho (A_{\rm N, e})$ and can be viewed as the Weyl function of
a quasi boundary triple for $(S_{\rm i} \oplus S_{\rm e})^*$, where $A_0$ in~\eqref{eq:SchroedingerOp} corresponds to the kernel of the
first boundary mapping.
We mention that a scalar analog of the function $\widetilde M$ in \eqref{widetildem} appears in connection with $\lambda$-dependent
Sturm--Liouville boundary value problems in  \cite{DLS87}
and in more general abstract form in \cite{DHMS00}, see also \cite{BLT13} for more details and references.

In the present setting Lemma~\ref{prop:qbtSchreod3} and Lemma~\ref{simpleagain} below
combined with the results in Section~\ref{sec:abstr} lead to an improvement of items (i)-(iv) in Theorem~\ref{thm:eigenSchroed}.
The assertions (v) and (vi) in Theorem~\ref{thm:eigenSchroed} remain valid with $M$ and $H^{1/2}(\Sigma)$ replaced by
$\widetilde M$ and $H^{1/2}(\Sigma)\times H^{3/2}(\Sigma)$, respectively, but will not be formulated again.

\begin{theorem}\label{thm:eigenSchroedDecoup}
Let $A_0$, $\Sigma$, and $\widetilde M$ be as above and let $\lambda\in\dR$.
Then the following assertions hold.
\begin{enumerate}
 \item $\lambda \in\sigma_{\rm p}(A_0)$ if and only if $R_\lambda \widetilde M:=$ \textup{s}-$\lim_{\eta \searrow 0} i \eta \widetilde M
 (\lambda + i \eta) \neq 0$; if the multiplicity of the eigenvalue $\lambda$ is finite then the mapping
 \begin{align}\label{eq:tauSchroeddeltaq}
  \tau : \ker (A_0 - \lambda) \to \ran R_\lambda \widetilde M, \quad u \mapsto
  \begin{pmatrix} u_{\rm i} |_\Sigma\\ \frac{\partial u_{\rm e}}{\partial \nu_{\cL_{\rm e}}} \big|_\Sigma\end{pmatrix},
 \end{align}
 is bijective; if the multiplicity of the eigenvalue $\lambda$ is infinite then the mapping
 \begin{align}\label{eq:taugenSchroeddeltaq}
  \tau : \ker (A_0 - \lambda) \to \cl_\tau \bigl( \ran R_\lambda \widetilde M\bigr), \quad u \mapsto
  \begin{pmatrix} u_{\rm i} |_\Sigma\\ \frac{\partial u_{\rm e}}{\partial \nu_{\cL_{\rm e}}} \big|_\Sigma\end{pmatrix},
 \end{align}
 is bijective, where $\cl_\tau$ denotes the closure in the normed space $\ran \tau$.
 \item $\lambda$ is an isolated eigenvalue of $A_0$ if and only if $\lambda$ is a pole in the strong sense of~$\widetilde M$.
 In this case \eqref{eq:tauSchroeddeltaq} and~\eqref{eq:taugenSchroeddeltaq} hold with $R_\lambda \widetilde M = \Res_\lambda \widetilde M$.
 \item $\lambda\in\rho(A_0)$ if and only if $\widetilde M$ can be continued analytically into $\lambda$.
 \item $\lambda\in\sigma_{\rm c}(A_0)$ if and only if \textup{s}-$\lim_{\eta \searrow 0} i \eta \widetilde M (\lambda + i \eta) = 0$ and
 $\widetilde M$
 cannot be continued analytically into $\lambda$.
\end{enumerate}
\end{theorem}

We provide a quasi boundary triple such that $\widetilde M$ in \eqref{widetildem} is the corresponding Weyl function.
As indicated above we make use of the densely defined, closed, symmetric operators
\begin{align*}
 S_j u_j = \cL_j u_j, \quad \dom S_j =
 \bigg\{ u_j \in H^2 (\Omega_j) : u_j |_\Sigma = \frac{\partial u_j}{\partial \nu_{\cL_j}} \Big|_\Sigma = 0 \bigg\},
\end{align*}
in $L^2 (\Omega_j)$ for $j = \rm i, e$, which appeared already the proof of Lemma~\ref{sonnenschein} and which are both simple.
Besides the operators $S_j$ also
the operators
\begin{align*}
 T_j u_j = \cL_j u_j, \quad \dom T_j = H^2 (\Omega_j),
\end{align*}
appear in the formulation of the next lemma.

\begin{lemma}\label{prop:qbtSchreod3}
The triple $\{ L^2 (\Sigma)\times L^2 (\Sigma), \widetilde\Gamma_0, \widetilde\Gamma_1\}$, where
$\widetilde\Gamma_0, \widetilde\Gamma_1 : \dom (T_{\rm i}\oplus T_{\rm e}) \to L^2 (\Sigma)\times L^2 (\Sigma)$ and
\begin{align*}
 \widetilde\Gamma_0 u = \begin{pmatrix}\frac{\partial u_{\rm i}}{\partial \nu_{\cL_{\rm i}}}
 \big|_\Sigma + \frac{\partial u_{\rm e}}{\partial \nu_{\cL_{\rm e}}} \big|_\Sigma \\ u_{\rm i} |_\Sigma-u_{\rm e} |_\Sigma\end{pmatrix} ,
 \quad \widetilde\Gamma_1 u = \begin{pmatrix} u_{\rm i} |_\Sigma\\ \frac{\partial u_{\rm e}}{\partial \nu_{\cL_{\rm e}}} \big|_\Sigma\end{pmatrix},
\end{align*}
is a quasi boundary triple for $S_{\rm i}^*\oplus S_{\rm e}^*$ such that
$(T_{\rm i}\oplus T_{\rm e})\upharpoonright \ker \widetilde\Gamma_0$ coincides with the operator $A_0$ in~\eqref{eq:SchroedingerOp} and $\ran\widetilde\Gamma_0= H^{1/2} (\Sigma)\times H^{3/2}(\Sigma)$.
For all $\lambda \in\rho(A_0)\cap \rho (A_{\rm D, i}) \cap \rho (A_{\rm N, e})$
the corresponding Weyl function coincides with the function $\widetilde M$ in \eqref{widetildem}.
\end{lemma}

\begin{proof}
The proof of Lemma~\ref{prop:qbtSchreod3} follows the same strategy as the proof of Lemma~\ref{prop:qbtSchreod} and some details are left
to the reader. Well known properties of traces of $H^2$-functions yield
$$\ran(\widetilde\Gamma_0,\widetilde\Gamma_1)^\top= \bigl(H^{1/2}(\Sigma)\times H^{3/2}(\Sigma)\bigr)\times\bigl(H^{3/2}(\Sigma)
\times H^{1/2}(\Sigma)\bigr),$$
which is dense in $(L^2 (\Sigma)\times L^2 (\Sigma))^2$. Moreover, $C_0^\infty(\dR^n\setminus\Sigma)$ is a dense subspace of $L^2 (\R^n)$ which is contained in $\ker\widetilde\Gamma_0\cap\ker\widetilde\Gamma_1$.
Green's identity implies that \eqref{eq:absGreen} holds, and as $H^2(\dR^n)$ is contained in $\ker\widetilde\Gamma_0$
the selfadjoint operator $A_0$ is contained in $(T_{\rm i}\oplus T_{\rm e})\upharpoonright \ker \widetilde\Gamma_0$. Hence the assumptions (i)-(iii) in
Proposition~\ref{prop:ratetheorem} are satisfied and it follows that $\{ L^2 (\Sigma)\times L^2 (\Sigma), \widetilde\Gamma_0, \widetilde\Gamma_1\}$
is a quasi boundary triple for $S_{\rm i}^*\oplus S_{\rm e}^*$ such that $A_0 =  (T_{\rm i}\oplus T_{\rm e})\upharpoonright \ker \widetilde\Gamma_0 $.

Let us verify that the corresponding Weyl function is given by $\widetilde M$ in \eqref{widetildem}. For this let
$\lambda \in\rho(A_0)\cap \rho (A_{\rm D, i}) \cap \rho (A_{\rm N, e})$ and let
$u_\lambda=u_{\lambda, \rm i}\oplus u_{\lambda, \rm e} \in\dom (T_{\rm i}\oplus T_{\rm e})$ be such that
$\cL_j u_{\lambda,j}=\lambda u_{\lambda,j}$, $j = \rm i, e$. Then we have
\begin{equation*}
 \begin{split}
  \begin{pmatrix} \Lambda_{\rm i}(\lambda) & 1 \\ 1 & -\Lambda_{\rm e}(\lambda)^{-1}\end{pmatrix}
  \widetilde\Gamma_1 u_\lambda &=
  \begin{pmatrix} \Lambda_{\rm i}(\lambda) & 1 \\ 1 & -\Lambda_{\rm e}(\lambda)^{-1}\end{pmatrix}
  \begin{pmatrix} u_{\lambda,{\rm i}} |_\Sigma\\ \frac{\partial u_{\lambda,{\rm e}}}{\partial \nu_{\cL_{\rm e}}} \big|_\Sigma\end{pmatrix}\\
  &=\begin{pmatrix} \Lambda_{\rm i}(\lambda) u_{\lambda,{\rm i}} |_\Sigma + \frac{\partial u_{\lambda,{\rm e}}}{\partial \nu_{\cL_{\rm e}}} \big|_\Sigma\\
     u_{\lambda,{\rm i}} |_\Sigma- \Lambda_{\rm e}(\lambda)^{-1}\frac{\partial u_{\lambda,{\rm e}}}{\partial \nu_{\cL_{\rm e}}} \big|_\Sigma
    \end{pmatrix}\\
   &= \begin{pmatrix}
     \frac{\partial u_{\lambda,{\rm i}}}{\partial \nu_{\cL_{\rm i}}} \big|_\Sigma+\frac{\partial u_{\lambda,{\rm e}}}{\partial \nu_{\cL_{\rm e}}}
     \big|_\Sigma \\ u_{\lambda,{\rm i}} |_\Sigma - u_{\lambda,{\rm e}} |_\Sigma
    \end{pmatrix}=\widetilde\Gamma_0 u_\lambda.
 \end{split}
\end{equation*}
By the definition of the Weyl function we obtain that the function $\widetilde M$ in \eqref{widetildem} coincides with the
Weyl function associated to the quasi boundary triple $\{ L^2 (\Sigma)\times L^2 (\Sigma), \widetilde\Gamma_0, \widetilde\Gamma_1\}$
for all
$\lambda \in\rho(A_0)\cap \rho (A_{\rm D, i}) \cap \rho (A_{\rm N, e})$.
\end{proof}

The next lemma is a direct consequence of the fact that the symmetric operators $S_{\rm i}$ and $S_{\rm e}$ are simple;
cf. \cite[Proposition 2.5]{BR12} and \cite[Proposition 2.2]{BR13}.

\begin{lemma}\label{simpleagain}
 The  symmetric operator $S_{\rm i}\oplus S_{\rm e}$ is simple.
\end{lemma}

\begin{proof}[{\bf Proof of Theorem~\ref{thm:eigenSchroedDecoup}}]
Let $\{L^2 (\Sigma)\times L^2(\Sigma), \widetilde\Gamma_0, \widetilde\Gamma_1 \}$ be the quasi boundary triple in Lemma~\ref{prop:qbtSchreod3}. Then 
$(T_{\rm i}\oplus T_{\rm e}) \upharpoonright \ker \widetilde\Gamma_0$ corresponds to the selfadjoint elliptic differential operator $A_0$ in~\eqref{eq:SchroedingerOp} and 
the associated Weyl function coincides with the operator function $\widetilde M$ in~\eqref{widetildem}. Taking Lemma~\ref{simpleagain} into account, 
item (i) follows from Corollary~\ref{thm:eigen} and items (ii)-(iv) are 
consequences of Theorem~\ref{thm:specTotal} and Proposition~\ref{prop:isolatedEV}.
\end{proof}

\end{document}